\documentclass[a4paper,twoside]{article}
\usepackage{a4}
\usepackage{amssymb}
\usepackage{amsmath}
\usepackage{upref}
\usepackage[active]{srcltx}
\usepackage[pagebackref,colorlinks,citecolor=blue,linkcolor=blue]{hyperref}
\usepackage[dvipsnames]{color}
\allowdisplaybreaks[2] % To make it possible for displaybreaks
\newcount\minutes \newcount\hours
\hours=\time
\divide\hours 60
\minutes=\hours
\multiply\minutes -60
\advance\minutes \time
\newcommand{\klockan}{\the\hours:{\ifnum\minutes<10 0\fi}\the\minutes}
\newcommand{\tid}{\today\ \klockan}
\newcommand{\prtid}{\smash{\raise 10mm \hbox{\LaTeX ed \tid}}}
\renewcommand{\prtid}{}
\makeatletter
\def\sectionmark#1{} %\markboth{{\sectnr #1}}{{\sectnr #1}}} %Journal
\def\subsectionmark#1{}
\newcommand{\sectnr}{\ifnum \c@secnumdepth >\z@
                 \thesection.\hskip 1em\relax \fi}
\def\@evenhead{\footnotesize\rm\thepage\hfil\leftmark\hfil\llap{\prtid}}
\def\@oddhead{\footnotesize\rm\rlap{\prtid}\hfil\rightmark\hfil\thepage}
\def\tableofcontents{\section*{Contents} %\@mkboth{Contents}{Contents}} %Journal
 \@starttoc{toc}}
\makeatother
\makeatletter
\def\@biblabel#1{#1.}
\makeatother
\makeatletter
\let\Thebibliography=\thebibliography
\renewcommand{\thebibliography}[1]{\def\@mkboth##1##2{}\Thebibliography{#1}
\addcontentsline{toc}{section}{References}
\frenchspacing % Maybe not needed
\setlength{\@topsep}{0pt}% Delete if extra space before list
\setlength{\itemsep}{0pt}%
\setlength{\parskip}{0pt plus 2pt}%
}
\makeatother
\makeatletter
\def\mdots@{\mathinner.\nonscript\!.%
 \ifx\next,.\else\ifx\next;.\else\ifx\next..\else
 \nonscript\!\mathinner.\fi\fi\fi}
\let\ldots\mdots@
\let\cdots\mdots@
\let\dotso\mdots@
\let\dotsb\mdots@
\let\dotsm\mdots@
\let\dotsc\mdots@
\def\vdots{\vbox{\baselineskip2.8\p@ \lineskiplimit\z@
    \kern6\p@\hbox{.}\hbox{.}\hbox{.}\kern3\p@}}
\def\ddots{\mathinner{\mkern1mu\raise8.6\p@\vbox{\kern7\p@\hbox{.}}%
    \raise5.8\p@\hbox{.}\raise3\p@\hbox{.}\mkern1mu}}
\makeatother
\makeatletter
\let\Enumerate=\enumerate
\renewcommand{\enumerate}{\Enumerate%
\setlength{\@topsep}{0pt}% Delete if extra space before list
\setlength{\itemsep}{0pt}%
\setlength{\parskip}{0pt plus 1pt}%
\renewcommand{\theenumi}{\textup{(\alph{enumi})}}%
\renewcommand{\labelenumi}{\theenumi}%
}
\let\endEnumerate=\endenumerate
\renewcommand{\endenumerate}{\endEnumerate\unskip}

\makeatother
\makeatletter
\def\@seccntformat#1{\csname the#1\endcsname.\quad}
\makeatother
\newcommand{\authortitle}[2]{\author{#1}\title{#2}\markboth{#1}{#2}}
\newcommand{\art}[6]{{\sc #1, \rm #2, \it #3\/ \bf #4 \rm (#5), \mbox{#6}.}}
\newcommand{\artin}[3]{{\rm #1, \rm #2, in #3.}}
\newcommand{\artnopt}[6]{{\sc #1, \rm #2, \it #3\/ \bf #4 \rm (#5), \mbox{#6}}}
\newcommand{\arttoappear}[3]{{\sc #1, \rm #2, to appear in \it #3}}
\newcommand{\auth}[2]{{#1, #2.}}
\newcommand{\book}[3]{{\sc #1, \it #2, \rm #3.}}
\newcommand{\AND}{{\rm and }}
\RequirePackage{amsthm}
\newtheoremstyle{descriptive}%
  {\topsep}   %{\medskipamount}          % Space above
  {\topsep}   %  {\medskipamount}          % Space below
  {\rmfamily} % Body font
  {}          % Indent
  {\bfseries} % Head font
  {.}         % Punctuation after thm head
  { }         % Space after thm head
  {}          % Thm head spec(?)
\newtheoremstyle{propositional}%
  {\topsep}   %  {\medskipamount}          % Space above
  {\topsep}   %  {\medskipamount}          % Space below
  {\itshape}  % Body font
  {}          % Indent
  {\bfseries} % Head font
  {.}         % Punctuation after thm head
  { }         % Space after thm head
  {}          % Thm head spec(?)
\theoremstyle{propositional}
\newtheorem{thm}{Theorem}[section]
\newtheorem{prop}[thm]{Proposition}
\newtheorem{lem}[thm]{Lemma}
\newtheorem{cor}[thm]{Corollary}
\theoremstyle{descriptive}
\newtheorem{deff}[thm]{Definition}
\newtheorem{remark}[thm]{Remark}
\newtheorem{example}[thm]{Example}
\makeatletter
\renewenvironment{proof}[1][\proofname]{\par
  \pushQED{\qed}%
  \normalfont
  \trivlist
  \item[\hskip\labelsep
        \itshape
    #1\@addpunct{.}]\ignorespaces
}{%
  \popQED\endtrivlist\@endpefalse
}
\makeatother
{\catcode`p =12 \catcode`t =12 \gdef\eeaa#1pt{#1}}      % Get slantfactor
\def\accentadjtext#1{\setbox0\hbox{$#1$}\kern   % Convert it with height
                \expandafter\eeaa\the\fontdimen1\textfont1 \ht0 }
\def\accentadjscript#1{\setbox0\hbox{$#1$}\kern % Convert it with height
                \expandafter\eeaa\the\fontdimen1\scriptfont1 \ht0 }
\def\accentadjscriptscript#1{\setbox0\hbox{$#1$}\kern   % Convert it with height
                \expandafter\eeaa\the\fontdimen1\scriptscriptfont1 \ht0 }
\def\accentadjtextback#1{\setbox0\hbox{$#1$}\kern       % Convert it with height
                -\expandafter\eeaa\the\fontdimen1\textfont1 \ht0 }
\def\accentadjscriptback#1{\setbox0\hbox{$#1$}\kern     % Convert it with height
                -\expandafter\eeaa\the\fontdimen1\scriptfont1 \ht0 }
\def\accentadjscriptscriptback#1{\setbox0\hbox{$#1$}\kern % Convert it with height
                -\expandafter\eeaa\the\fontdimen1\scriptscriptfont1 \ht0 }
\def\itoverline#1{{\mathsurround0pt\mathchoice
        {\rlap{$\accentadjtext{\displaystyle #1}
                \accentadjtext{\vrule height1.593pt}
                \overline{\phantom{\displaystyle #1}
                \accentadjtextback{\displaystyle #1}}$}{#1}}
        {\rlap{$\accentadjtext{\textstyle #1}
                \accentadjtext{\vrule height1.593pt}
                \overline{\phantom{\textstyle #1}
                \accentadjtextback{\textstyle #1}}$}{#1}}
        {\rlap{$\accentadjscript{\scriptstyle #1}
                \accentadjscript{\vrule height1.593pt}
                \overline{\phantom{\scriptstyle #1}
                \accentadjscriptback{\scriptstyle #1}}$}{#1}}
        {\rlap{$\accentadjscriptscript{\scriptscriptstyle #1}
                \accentadjscriptscript{\vrule height1.593pt}
                \overline{\phantom{\scriptscriptstyle #1}
                \accentadjscriptscriptback{\scriptscriptstyle #1}}$}{#1}}}}
\def\itunderline#1{{\mathsurround0pt\mathchoice
        {\rlap{$\underline{\phantom{\displaystyle #1}
                \accentadjtextback{\displaystyle #1}}$}{#1}}
        {\rlap{$\underline{\phantom{\textstyle #1}
                \accentadjtextback{\textstyle #1}}$}{#1}}
        {\rlap{$\underline{\phantom{\scriptstyle #1}
                \accentadjscriptback{\scriptstyle #1}}$}{#1}}
        {\rlap{$\underline{\phantom{\scriptscriptstyle #1}
                \accentadjscriptscriptback{\scriptscriptstyle #1}}$}{#1}}}}
\newcommand{\limplus}{{\mathchoice{\vcenter{\hbox{$\scriptstyle +$}}}
  {\vcenter{\hbox{$\scriptstyle +$}}}
  {\vcenter{\hbox{$\scriptscriptstyle +$}}}
  {\vcenter{\hbox{$\scriptscriptstyle +$}}}
}}
\def\vint{\mathop{\mathchoice%
          {\setbox0\hbox{$\displaystyle\intop$}\kern 0.22\wd0%
           \vcenter{\hrule width 0.6\wd0}\kern -0.82\wd0}%
          {\setbox0\hbox{$\textstyle\intop$}\kern 0.2\wd0%
           \vcenter{\hrule width 0.6\wd0}\kern -0.8\wd0}%
          {\setbox0\hbox{$\scriptstyle\intop$}\kern 0.2\wd0%
           \vcenter{\hrule width 0.6\wd0}\kern -0.8\wd0}%
          {\setbox0\hbox{$\scriptscriptstyle\intop$}\kern 0.2\wd0%
           \vcenter{\hrule width 0.6\wd0}\kern -0.8\wd0}}%
          \mathopen{}\int}
\numberwithin{equation}{section}
\newenvironment{ack}{\medskip{\it Acknowledgement.}}{}
\renewcommand{\emptyset}{\varnothing}
\DeclareMathOperator{\Div}{div}
\renewcommand{\phi}{\varphi}
\newcommand{\al}{\alpha}
\newcommand{\Ga}{\Gamma}
\newcommand{\eps}{\varepsilon}

\newcommand{\R}{\mathbf{R}}
\newcommand{\C}{\mathbf{C}}
\newcommand{\Z}{\mathbf{Z}}
\newcommand{\Sp}{\mathbf{S}}%
\newcommand{\eR}{{\overline{\R}}}
\newcommand{\Rnp}{\R^n_\limplus}

\newcommand{\p}{{$p\mspace{1mu}$}}

\newcommand{\Np}{N^{1,p}}
\newcommand{\Nploc}{N^{1,p}\loc}
\newcommand{\Dp}{D^p}
\newcommand{\Dploc}{D^{p}\loc}
\newcommand{\Lploc}{L^{p}\loc}

\newcommand{\UU}{\mathcal{U}}%
\newcommand{\bdy}{\partial}
\newcommand{\bdry}{\partial}
\newcommand{\bdhat}{\widehat{\partial}}
\newcommand{\setm}{\setminus}
\newcommand{\ga}{\gamma}
\newcommand{\Om}{\Omega}
\newcommand{\sig}{\sigma}
\newcommand{\dha}{\hat{d}}
\newcommand{\hh}{\hat{h}}
\newcommand{\sha}{\hat{s}_a}
\newcommand{\muh}{\hat{\mu}}
\newcommand{\muha}{\hat{\mu}}
\newcommand{\uh}{\hat{u}}
\newcommand{\gh}{\hat{g}}
\newcommand{\tf}{\tilde{f}}
\newcommand{\Xdot}{{\widehat{X}}}
\newcommand{\Omhat}{{\widehat{\Om}}}
\newcommand{\Bhat}{{\widehat{B}}}
\DeclareMathOperator{\Lip}{Lip}
\newcommand{\Lipc}{{\Lip_c}}
\DeclareMathOperator*{\essliminf}{ess\,lim\,inf}
\newcommand{\loc}{_{\rm loc}}
\newcommand{\Cp}{{C_p}}
\newcommand{\CpX}{\Cp} 
\newcommand{\CpXdot}{\Cphat} 
\newcommand{\Cphat}{{\widehat{C}_p}}
\newcommand{\lSo}{\itunderline{S}_0}
\newcommand{\uSo}{\itoverline{S}_0}

\DeclareMathOperator{\capp}{cap}
\newcommand{\cphat}{\capp^\Xdot_p}

\newcommand{\simge}{\gtrsim}
\newcommand{\simle}{\lesssim}

\newcommand{\Hp}{P}                 % upper=lower solution
\newcommand{\uHp}{\itoverline{P}}  
\newcommand{\uP}{\itoverline{P}}     
\newcommand{\lP}{\itunderline{P}}

\DeclareMathOperator*{\essinf}{ess\,inf}
\DeclareMathOperator{\diam}{diam}
\DeclareMathOperator{\Mod}{Mod}
\newcommand{\Modp}{{\Mod_p}}
\newcommand{\pharm}{{\omega}_p} 

\begin{document}
%
% \authortitle sets headlines and also \author and \title, but
% the latter can be changed afterwards if needed for page 1
%
\authortitle{Anders Bj\"orn, Jana Bj\"orn and Xining Li}
{Sphericalization and
\p-harmonic functions on unbounded domains in Ahlfors regular spaces}
\title{Sphericalization and
\p-harmonic functions on unbounded domains in Ahlfors regular metric spaces}
\author{
Anders Bj\"orn \\
\it\small Department of Mathematics, Link\"oping University, \\
\it\small SE-581 83 Link\"oping, Sweden\/{\rm ;}
\it \small anders.bjorn@liu.se
\\
\\
Jana Bj\"orn \\
\it\small Department of Mathematics, Link\"oping University, \\
\it\small SE-581 83 Link\"oping, Sweden\/{\rm ;}
\it \small jana.bjorn@liu.se
\\
\\
Xining Li \\
\it\small Department of Mathematics, Sun Yat-sen University, \\
\it \small  Guangzhou, 510275, P. R. China\/{\rm ;}
E-mail address: \it \small lixining3@mail.sysu.edu.cn
}

%\date{Preliminary version, \today}
\date{}
\maketitle

\noindent{\small
{\bf Abstract}.
We use sphericalization to study 
the Dirichlet problem, Perron solutions 
and boundary regularity for \p-harmonic
functions on unbounded sets in Ahlfors regular metric spaces.
Boundary regularity for the point at infinity is given
special attention.
In particular, we allow for several ``approach directions''
towards infinity and take into account the massiveness of their
complements.

In 2005, Llorente--Manfredi--Wu 
showed that the 
\p-harmonic measure  on the upper half space
$\Rnp$, $n \ge 2$, is not subadditive on null sets 
when $p \ne 2$.
Using their result and spherical inversion, we create similar
bounded examples in the unit ball 
$\mathbf{B}\subset\R^n$ showing that the $n$-harmonic
measure is not subadditive on null sets when $n \ge 3$, 
and neither are the \p-harmonic measures in $\mathbf{B}$
generated by certain weights depending on $p\ne2$ and $n\ge2$.
}

\bigskip
\noindent
{\small \emph{Key words and phrases}:
Ahlfors regular metric space,
boundary regularity,
Dirichlet problem,
Muckenhoupt $A_p$ weight,
Perron solution,
\p-harmonic function,
\p-harmonic measure,
Poincar\'e inequality,
sphericalization.
}


\section{Introduction}

In this paper we use sphericalization to
study Perron solutions and boundary regularity
for \p-harmonic functions on unbounded sets in metric measure spaces.
On unweighted $\R^n$, a \emph{\p-harmonic function} in 
an open set $\Om \subset \R^n$
is a continuous
weak solution $u$ of the \p-Laplace equation
\[
    \Delta_p u = \Div(|\nabla u|^{p-2} \nabla u)=0.
\]
Equivalently, $u$ is a continuous minimizer of the \p-energy integral
among functions with the same boundary values, i.e.\
\[
    \int_{\phi \ne 0}  |\nabla u|^p \, dx
    \le \int_{\phi \ne 0}  |\nabla (u+\phi)|^p \, dx
    \quad
    \text{for all } \phi \in \Lipc(\Om).
\] 
This latter definition is suitable for generalization to
metric spaces, using \p-weak upper gradients.

The setting considered in this paper is a complete
metric space $(X,d)$ equipped with an Ahlfors $Q$-regular measure $\mu$, $Q>1$,
supporting a $q$-Poincar\'e inequality.
The precise relations between $p$, $Q$ and $q$ are
given in \eqref{eq-ass-pQ} and \eqref{eq-ass-pQ-PI}
and make it possible to ``sphericalize'' $X$ into a new bounded metric
space with suitable properties, where the theory of \p-harmonic functions
is well-developed.

Our primary object of study is the Dirichlet problem,
which given $f:\bdy \Om \to \R$ asks for a \p-harmonic function
having $f$ as boundary values in some weak sense. 
We use the Perron method to solve this problem,
which always provides two solutions $\lP f \le \uP f$ of
the Dirichlet problem.
When they coincide
we denote the common solution by $\Hp f$
and $f$ is called \emph{resolutive}.
Under the assumptions in this paper, continuous functions are resolutive.
A boundary point $x_0 \in \bdy \Om$ is \emph{regular}
if 
\[
\lim_{\Om \ni y \to x} Pf(y)=f(x) 
\quad \text{for all } f \in C(\bdy \Om).
\]

Perron solutions for \p-harmonic functions on \emph{bounded}
open sets $\Om$ in metric spaces were first  studied in
Bj\"orn--Bj\"orn--Shanmugalingam~\cite{BBS2}.
Boundary regularity 
for \p-harmonic functions on bounded open sets in metric spaces
was studied even earlier in 
Bj\"orn~\cite{Bj02},
Bj\"orn--MacManus--Shanmugalingam~\cite{BMS}
and Bj\"orn--Bj\"orn--Shanmugalingam~\cite{BBS};
and a rather extensive study 
was carried out in Bj\"orn--Bj\"orn~\cite{BB}.

In this paper, our main interest lies in proving
resolutivity
and boundary regularity in  \emph{unbounded}
$\Om$ by  using sphericalization. 
The sphericalization technique makes it possible to map
the unbounded space $X$ into a bounded metric space,
while preserving the \p-harmonic functions, 
in the spirit of the Kelvin transform.
We can then appeal to the earlier results for bounded sets and 
``map'' them back to the original unbounded situation.
When $\Om$ is unbounded we consider its boundary within the one-point
compactification of $X$, so $\infty$ is 
included in the boundary.

Perron solutions for \p-harmonic functions on unbounded
open sets  in metric spaces have been
studied by Hansevi~\cite{Hansevi2} who realised
that there were substantial additional complications in the
unbounded case compared with the bounded case.
There is some overlap between his and our results.
To our best knowledge, boundary regularity for \p-harmonic 
functions on unbounded open sets 
in metric spaces has not been studied before,
beyond weighted $\R^n$.
However, regularity of the boundary at $\infty$
for global \p-harmonic functions on certain Cartan--Hadamard manifolds
and Gromov hyperbolic spaces
was considered in e.g.\ Holopainen~\cite{Holo02}, 
Holopainen--Lang--V\"ah\"akangas~\cite{HoloLV07}
and Holopainen--V\"ah\"akangas~\cite{HoloV07}.

Li--Shanmugalingam~\cite{LiShan} showed that (under suitable conditions)
sphericalization preserves Ahlfors regularity and Poincar\'e inequalities.
These significant results are our starting point.
The measure $\mu_a$, with which they equipped the sphericalization,
does however not preserve the \p-energy, and we are therefore
forced to equip the sphericalization with a different 
measure $\muha$, generated by a suitable power weight.
Using (a metric space version of) the theory of Muckenhoupt $A_p$ weights we
can show that the sphericalization equipped with $\muha$ supports
a \p-Poincar\'e inequality, provided the original space
supports a $q$-Poincar\'e inequality, with $q$ given by
\eqref{eq-ass-pQ-PI}.
This whole machinery is carried out in Sections~\ref{S:Prelim}--\ref{sect-PI}.

In Section~\ref{sect-pharm} we are then ready to start 
our study of \p-harmonic functions and especially Perron solutions
on unbounded sets. 
We obtain several 
resolutivity and perturbation results,
see Theorems~\ref{thm-resol-C} and~\ref{thm-Np-res}.
We also show that boundary regularity is a local property,
that it can be characterized using barriers, and that the Kellogg property
holds, i.e.\ the set of irregular boundary points has capacity zero.

Since regularity is a local property, many of the results
on boundary regularity for bounded sets, such as the Wiener criterion,
carry over directly to finite boundary points of unbounded sets.
The point at $\infty$, however, requires special attention
and is studied in Section~\ref{sect-infty}.

In unweighted $\R^n$, $n\ge2$, with $p > n/2$ (in particular in $\R^2$ 
for any $p>1$) we get a number of new results.
The resolutivity of continuous functions, the Kellogg property and the barrier
characterization have been known before in this setting, see
Kilpel\"ainen~\cite[Theorems~1.5, 1.10 and Corollary~5.6]{Kilp89}.
The obtained resolutivity results for Newtonian (and Dirichlet) 
functions are new when $n/2 < p<n$, while 
for $p\ge n$ they were proved in
Hansevi~\cite{Hansevi2} using different methods.
Theorem~\ref{thm-uniq-bdd-pharm} is new for all $p > n/2$,
although \cite{Hansevi2} contains a weaker result for 
unbounded \p-parabolic sets.
We also obtain several
new characterizations of boundary regularity in 
unbounded sets, corresponding to the results in 
Bj\"orn--Bj\"orn~\cite[Theorem~6.1]{BB}; see also 
Heinonen--Kilpel\"ainen--Martio~\cite[Chapter~9]{HeKiMa}.

Since the \p-energy is preserved, also quasiminimizers
are preserved in just the same way as \p-harmonic functions.
Hence, many earlier boundary regularity results for quasiminimizers
generalize from bounded to unbounded sets, see 
the end of Section~\ref{sect-pharm}.
It also follows that sphericalization is a quasiconformal mapping,
by Theorem~4.1 in  Korte--Marola--Shanmugalingam~\cite{KoMaSh}.

In unbounded domains, the point at infinity can often be approached from
different directions, e.g.\ in an infinite strip or cylinder.
Also on bounded domains, rather than
using the given metric boundary, there are many situations
where one would like to distinguish between
different directions towards a boundary point. 
For example, in the slit disc, where a horizontal ray is removed 
from the disc, 
it is natural to consider different boundary values along
the ray from above and from below.
Bj\"orn--Bj\"orn--Shan\-mu\-ga\-lin\-gam~\cite{BBSdir} carried
out such a study using the Mazurkiewicz metric on
bounded domains in metric spaces.
We are now able to 
transfer these results to unbounded domains,
and we explain in particular what happens at infinity in
Section~\ref{sect-infty}.
We also provide an example where the point at $\infty$
corresponds to \emph{uncountably} many directions, i.e.\
boundary points with respect to the Mazurkiewicz metric, and yet
behaves well for the Perron method, see Example~\ref{ex-uncountable}.

Llorente--Manfredi--Wu~\cite{LoMaWu}, using 
a very sophisticated argument due to Wolff~\cite{wolff},
showed that for any $p \ne 2$ and $n \ge 2$, there are
sets 
$A_1,A_2 \subset \R^{n-1}$ such that
\begin{equation}   \label{eq-ex-nonadditive}
   \pharm(A_1;\Rnp)= \pharm(A_2;\Rnp) =0 
  < \pharm(A_1 \cup A_2;\Rnp),
\end{equation}
where $\pharm(\,\cdot\,;\Rnp)$ denotes the \p-harmonic
measure with respect to the upper half space $\Rnp$.
Despite its name, $\pharm(\,\cdot\,;\Rnp)$ is evidently not a measure, 
but rather a nonlinear generalization of the harmonic measure, and
\eqref{eq-ex-nonadditive} shows that it is not even subadditive.
In Section~\ref{sect-pharm-meas} we transfer the examples 
from~\cite{LoMaWu} to the unit ball in $\R^n$ and thus
create similar bounded examples with respect to 
a weighted measure depending on $p$ and $n$.
In particular,  the measure is the usual Lebesgue measure 
(\emph{without} weight) when $p=n$, so we obtain an analogue of 
\eqref{eq-ex-nonadditive} for the usual $n$-harmonic measure for the
$n$-Laplacian in the unit ball.

\begin{ack}
The first two authors were supported by the Swedish Research Council
grants 2016-03424 and 621-2014-3974, respectively. 
The third author was supported by NSFC grant 11701582.
Part of the
research was done during two visits of the third author to 
Link\"oping University in 2016 and 2018.
\end{ack}

\section{Metric spaces and power weights}
\label{S:Prelim}

We assume throughout the paper
that $1<p<\infty$ and that $X=(X,d,\mu)$ is a metric space equipped
with a metric $d$ and a positive complete  Borel  measure $\mu$
such that $0<\mu(B)<\infty$ for all  
balls $B \subset X$.
Additional
standing assumptions will be given at the
beginning of Sections~\ref{sect-sphericalization} and~\ref{sect-PI}.
Proofs of the results in 
this section can be found in the monographs
Bj\"orn--Bj\"orn~\cite{BBbook} and
Heinonen--Koskela--Shanmugalingam--Tyson~\cite{HKSTbook}.

The measure  $\mu$  is \emph{doubling} if
there exists $C>0$ such that for all balls
$B=B(x_0,r):=\{x\in X: d(x,x_0)<r\}$ in~$X$,
\begin{equation*}
        0 < \mu(2B) \le C \mu(B) < \infty.
\end{equation*}
Here and elsewhere  
$\lambda B=B(x_0,\lambda r)$.
A metric space with a doubling measure is proper\/
\textup{(}i.e.\ closed bounded subsets are compact\/\textup{)}
if and only if it is complete.

A \emph{curve} is a continuous mapping from an interval,
and a \emph{rectifiable} curve is a curve with finite length.
Unless otherwise mentioned, we
 will only consider curves which are nonconstant, compact
and 
rectifiable, and thus each curve can 
be parameterized by its arc length $ds$. 
For a  family $\Gamma$ of curves in $X$,
we define its \emph{\p-modulus}
\[
         \Modp(\Gamma):=\inf \int_X \rho^p \, d\mu,
\]
where the infimum is taken over all 
Borel functions $\rho\ge0$
such that $\int_\ga \rho \,ds \ge 1$ for all $\ga \in \Gamma$.
A property is said to hold for \emph{\p-almost every curve}
if it fails only for a curve family $\Ga$ with zero \p-modulus.
Following Heinonen--Kos\-ke\-la~\cite{HeKo98},
we introduce upper gradients 
as follows 
(they called them very weak gradients).

\begin{deff} \label{deff-ug}
A Borel function $g : X \to [0,\infty]$  is an \emph{upper gradient} 
of a function $u: X \to \eR:=[-\infty,\infty]$
if for all  curves  
$\gamma : [0,l_{\gamma}] \to X$,
\begin{equation} \label{ug-cond}
|u(\gamma(0)) - u(\gamma(l_{\gamma}))| \le \int_{\gamma} g\,ds,
\end{equation}
where the left-hand side is interpreted as
$\infty$ whenever at least one of the terms therein is infinite.
If $g: X \to [0,\infty]$ is measurable 
and \eqref{ug-cond} holds for \p-almost every curve,
then $g$ is a \emph{\p-weak upper gradient} of~$u$. 
\end{deff}

The \p-weak upper gradients were introduced in
Koskela--MacManus~\cite{KoMc}. 
It was also shown therein
that if $g \in \Lploc(X)$ is a \p-weak upper gradient of $u$,
then one can find a sequence $\{g_j\}_{j=1}^\infty$
of upper gradients of $u$ such that $\|g_j-g\|_{L^p(X)}  \to 0$.
If $u$ has an upper gradient in $\Lploc(X)$, then
it has an a.e.\ unique \emph{minimal \p-weak upper gradient} $g_u \in \Lploc(X)$
in the sense that $g_u \le g$ a.e.\ for every \p-weak upper gradient 
$g \in \Lploc(X)$ of $u$, see Shan\-mu\-ga\-lin\-gam~\cite{Sh-harm}.

Together with the doubling property defined above, the following
Poincar\'e inequality is often a standard assumption on metric spaces.

\begin{deff} \label{def.PI.}
We say that $X$ (or $\mu$) supports a \emph{$q$-Poincar\'e inequality},
$q \ge 1$,  if
there exist constants $C>0$ and $\lambda \ge 1$
such that for all balls $B \subset X$,
all integrable functions $u$ on $X$ and all ($q$-weak) 
upper gradients $g$ of $u$,
\[ 
        \vint_{B} |u-u_B| \,d\mu
        \le C \diam(B) \biggl( \vint_{\lambda B} g^{q} \,d\mu \biggr)^{1/q},
\] 
where $u_B :=\vint_B u \,d\mu := \int_B u\, d\mu/\mu(B)$.
\end{deff}

As is customary, we say that $A \simle B$ (and equivalently
$B \simge A$), if there is a constant $C>0$
(independent of the variables that $A$ and $B$ are functions of)
such that $A\le C B$. We also write $A\simeq B$ if $A \simle B \simle A$.

With this notation,  $\mu$ is \emph{Ahlfors $Q$-regular}
if
\[ 
     \mu(B(x,r)) \simeq r^Q.
\] 

Ahlfors regularity is a relatively strong assumption on the measure.
At the same time, it is easily verified that
doubling measures in connected spaces satisfy the one-sided
estimates
\begin{equation*}   
\biggl( \frac{r'}{r} \biggr)^s \simle
\frac{\mu(B(x',r'))}{\mu(B(x,r))} \simle \biggl( \frac{r'}{r} \biggr)^\sig
\end{equation*}   
for some $0<\sig\le s<\infty$, all $x\in X$, $x'\in B(x,r)$ 
and all  $0<r'\le r\le 2 \diam X$.

The following proposition makes it possible to construct new well-behaved
measures from old ones.

\begin{prop}  \label{prop-power-Ap}
Let $X$ be a metric space equipped with a doubling measure $\mu$ such that
for some $\sig>0$, $c\in X$ and all $0<r'\le r\le 2 \diam X$,
\begin{equation}   \label{eq-def-sig}
\frac{\mu(B(c,r'))}{\mu(B(c,r))} \simle \biggl( \frac{r'}{r} \biggr)^\sig.
\end{equation}   
Let $\al\in\R$ and $w(x)=d(x,c)^\al$.
Then the following are true for all balls $B  \subset X$, with 
comparison constants depending on the one in \eqref{eq-def-sig}, 
as well as on $\al$ and $\sig$\textup{:}
\begin{enumerate}
\item If $p>1$ and $-\sig<\al <\sig(p-1)$, then
\begin{equation}    \label{eq-w-Ap}
\vint_B w\,d\mu \biggl( \vint_B w^{1/(1-p)}\,d\mu \biggr)^{p-1}
\simle 1.
\end{equation}
\item If $-\sig<\al\le0$, then
\begin{equation}    \label{eq-w-A1}
\vint_B w\,d\mu \simle \essinf_B w.
\end{equation}
\end{enumerate}
\end{prop}

If \eqref{eq-w-Ap} holds, we say that $w$ is an \emph{$A_p$ weight} with respect 
to $\mu$,
while if \eqref{eq-w-A1} holds then $w$ is an \emph{$A_1$ weight}. We
also write $w \in A_p(\mu)$, $p \ge 1$, in these cases.

\begin{proof}
We shall distinguish between three types of balls:

1.\ Let $B=B(c,r)$, $0<r\le 2\diam X$. Then for all $\al>-\sig$,
\begin{align}
\int_B d(x,c)^\al\,d\mu 
&\simeq \sum_{j=0}^\infty (2^{-j}r)^\al \mu(2^{-j}B \setm 2^{-j-1}B) \nonumber\\
&\simle \sum_{j=0}^\infty (2^{-j}r)^\al (2^{-j})^\sig \mu(B)
\simeq r^\al \mu(B).   \label{eq-al-ge-minus-sig}
\end{align}
Replacing $\al$ by $\al/(1-p)>-\sig$ in~\eqref{eq-al-ge-minus-sig}
we obtain
\[
\biggl( \vint_B w^{1/(1-p)}\,d\mu \biggr)^{p-1} \simle (r^{\al/(1-p)})^{p-1} = r^{-\al}.
\]
Together with~\eqref{eq-al-ge-minus-sig}, this yields~\eqref{eq-w-Ap}
for $B=B(c,r)$.
To prove~\eqref{eq-w-A1} for such $B$, it suffices to note that for all 
$\al\le0$,
\[
\essinf_B d(x,c)^\al \ge r^\al.
\]

2.\ If $B=B(z,r)$ and $r>\tfrac12 d(z,c)$ then $B\subset B(c,3r)$ and
$\mu(B)\simeq\mu(B(c,3r))$, by the doubling property of $\mu$.
We can therefore replace $B$ by $B(c,3r)$ in~\eqref{eq-w-Ap} 
and~\eqref{eq-w-A1} as follows, using case~1,
\[
\vint_B w\,d\mu \biggl( \vint_B w^{1/(1-p)}\,d\mu \biggr)^{p-1}
\simle \vint_{B(c,3r)} w\,d\mu \biggl( \vint_{B(c,3r)} w^{1/(1-p)}\,d\mu 
\biggr)^{p-1} \simle 1
\]
and
\[
\vint_B w\,d\mu \simle \vint_{B(c,3r)} w\,d\mu \simle \essinf_{B(c,3r)} w 
\le \essinf_B w.
\]

3.\ If $B=B(z,r)$ and $0<r\le\tfrac12 d(z,c)$ then
$w(x)=
d(x,c)^\al\simeq d(z,c)^\al$ for all $x\in B$ and hence
\[
\vint_B w\,d\mu \simeq d(z,c)^\al
\quad \text{and} \quad
\vint_B w^{1/(1-p)}\,d\mu \simeq d(z,c)^{\al/(1-p)},
\]
from which~\eqref{eq-w-Ap} follows.
Similarly,
\[
\vint_B w\,d\mu \simeq d(z,c)^\al \simeq \essinf_B w.\qedhere
\] 
\end{proof}

Proposition~\ref{prop-power-Ap} can be combined with 
Theorem~4 in Bj\"orn~\cite{JBpower} 
(whose 
proof works also
in metric spaces) to obtain the following result.

\begin{cor}   \label{cor-d-al-PI}
Assume that $\mu$ is doubling, satisfies~\eqref{eq-def-sig} 
and supports a $q$-Poincar\'e inequality on $X$.
Then the measure $d\nu=d(x,c)^\al\,d\mu$, with 
$-\sig<\al \le0$, 
is also doubling and supports a $q$-Poincar\'e inequality.
For $\al >0$ 
it is doubling and supports
a $q'$-Poincar\'e inequality for every $q'>q(1+\al/\sig)$.
\end{cor}

\section{Sphericalization}
\label{sect-sphericalization}

\emph{From now on we assume that $(X,d,\mu)$ is complete and unbounded.}

\medskip

Following Li--Shanmugalingam~\cite{LiShan} we will now define
the sphericalization of $X$.
Let $\Xdot=X\cup\{\infty\}$ be the one-point compactification of $X$.
We also fix a base point $a \in X$ from now on.
Define $d_a,\dha:\Xdot \times \Xdot \to [0,\infty)$ by
\[
  d_a(x,y) = d_a(y,x)=\begin{cases}
    \dfrac {d(x,y)}{(1+d(x,a))(1+d(y,a))},&\text{if }x,y\in X,\\[3mm]
              \dfrac{1}{1+d(x,a)},&\text{if }x\in X \text{ and } y=\infty,\\[3mm]
                                           0,&\text{if }x=y=\infty,
                    \end{cases}
\]
and
\begin{equation} \label{eq-metric-dh}
  \dha(x,y)=\inf_{(x=x_0,x_1,\ldots,x_k=y)}  \sum_{j=1}^k d_a(x_j,x_{j-1}),
\end{equation}
where the infimum is over all finite sequences $(x=x_0,x_1,\ldots,x_k=y)$.
This makes $\dha$ into a metric on $\Xdot$, and $(\Xdot,\dha)$ 
is the \emph{sphericalization} of $(X,d)$.
Moreover, 
\begin{equation}   \label{eq-comp-d-dhat}
\tfrac{1}{4} d_a(x,y) \le \dha(x,y) \le d_a(x,y),
\end{equation}
see (3.2) in Buckley--Herron--Xie~\cite{BuckleyHerronXie}
and the proof of Lemma~2.2 in Bonk--Kleiner~\cite{BonkKleiner02}.
Note that $d_a$ is in general not a metric since the triangle inequality
may fail for it.
We will denote balls in $(\Xdot,\dha)$ by $\Bhat(x,r)$.

In Li--Shanmugalingam~\cite{LiShan}, 
the sphericalization $(\Xdot,\dha)$ is equipped
with the measure $\mu_a$ defined
by 
\[
\mu_a(A)= \int_{A \setm \{\infty\}}  \frac{d\mu(x)}{\mu(B(a,1+d(x,a)))^2}.
\]

\begin{prop}
It is always true that $\mu_a(\Xdot) < \infty$.
\end{prop}

\begin{proof}
Let $b_j=\mu(B(a,j))$ (with $b_0=0$).
Then
\[
   \mu_a(\Xdot) \le \sum_{j=0}^\infty \frac{b_{j+1} - b_j}{b_{j+1}^2}
   \le \frac{1}{b_1} + \sum_{j=1}^\infty \frac{b_{j+1} - b_j}{b_j b_{j+1}}
   = \frac{1}{b_1} 
     + \sum_{j=1}^\infty \biggl(\frac{1}{b_{j}}  - \frac{1}{b_{j+1}}
     \biggr)
   \le   \frac{2}{b_1}.
   \qedhere
\]
\end{proof}

In this paper it will be more useful
to equip
$(\Xdot,\dha)$ with the measure $\muha$ defined by
\[ 
\muha(A)= \int_{A \setm \{\infty\}}  \frac{d\mu(x)}{(1+d(x,a))^{2p}}.
\] 
In order to use the results from \cite{LiShan} we will 
need to carefully study the connections between the
measures $\mu_a$ and $\muha$, which we do in Section~\ref{sect-PI}.
For the rest of this section we will only use the measure
$\muha$ on $\Xdot$.
Note that $\muha$ depends on $p$, even though this is not
made explicit in the notation.

The measure  $\muha(\Xdot)$ can be either finite or infinite.
Strictly speaking, as $\Xdot$ is bounded,
to fall within the scope of the theory considered
e.g.\ in \cite{BBbook}
we would need to require that $\muha(\Xdot)<\infty$,
but the results in the rest of this section,
as well as in Section~\ref{sect-Newt},
remain valid
also in the case when $\muha(\Xdot)=\infty$.
Under the assumptions at the beginning of Section~\ref{sect-PI} it follows
from Proposition~\ref{prop-PI-for-muh} that $\muha(\Xdot)<\infty$.

If $\ga :[0,1] \to X$ is a (not necessarily rectifiable) curve, then we can
consider its length with respect to $d$ and with respect to $\dha$.
It is quite easy (cf.\ Li--Shanmugalingam~\cite{LiShan})
to see that the arc lengths $ds$ and $d\sha$ with respect to $d$ 
and $\dha$, respectively, are related by
\begin{equation} \label{eq-dsha}
   d\sha(x)= \frac{ds(x)}{(1+d(x,a))^2}.
\end{equation}
As $\ga([0,1])$ is compact it follows that $\ga$ is rectifiable with respect to $d$
if and only if it is rectifiable with respect to $\dha$.

\begin{lem} \label{lem-Modp}
Let $\Ga$ be a collection of rectifiable curves on $X$.
Then 
\[
 \Modp(\Ga;X,d,\mu)=\Modp(\Ga;\Xdot,\dha,\muha).
\]
\end{lem}

\begin{proof}
Let $\rho$ be a nonnegative Borel function which is admissible in the 
definition of $ \Modp(\Ga;X,d,\mu)$,
i.e.\ 
 $\int_{\ga}\rho\,ds\ge 1$ for all $\ga\in\Gamma$.
 Let $\hat{\rho}(x)=\rho(x)(1+d(x,a))^2$.
Then, by \eqref{eq-dsha},
\[
    \int_{\ga} \hat{\rho}\,d\sha 
    = \int_{\ga} \rho\,ds \ge 1,
\]
and thus $\hat{\rho}$ is admissible in the definition of 
$\Modp(\Ga;\Xdot,\dha,\muha)$.
Moreover,
\[
    \int_{\Xdot} \hat{\rho}^p \,d\muha
    = \int_X \rho(x)^p (1+d(x,a))^{2p}  \frac{d\mu(x)}{(1+d(x,a))^{2p}}
    = \int_X \rho^p \, d\mu.
\]
Taking infimum over all such $\rho$ shows that
$\Modp(\Ga;\Xdot,\dha,\muha) \le \Modp(\Ga;X,d,\mu)$.
The converse inequality is shown similarly.
\end{proof}

\begin{lem} \label{lem-pwug}
Let $\Om \subset X$ be open,
$u:\Om \to \eR$ be a function, $g:\Om \to [0,\infty]$
be measurable, and
\[
    \gh(x)=g(x)(1+d(x,a))^2, 
    \quad  x \in \Om.
\]
Then $g$ is a \p-weak upper gradient of $u$ in $\Om$ with respect to
$(d,\mu)$ if and only if $\gh$ is a 
\p-weak upper gradient of $u$ in $\Om$ with respect to
$(\dha,\muha)$.
\end{lem}

Observe that measurability is the same with respect to $\mu$ and $\muha$.

\begin{proof}
Assume that $g$ is a \p-weak upper gradient of $u$ with respect to
$(d,\mu)$.
Let $\Ga$ be the family of exceptional curves in $\Om$ for which \eqref{ug-cond} fails.
Then $\Modp(\Ga;X,d,\mu)=0$.
Let $\ga:[0,l_\ga] \to \Om$ be a curve not in $\Ga$.
Then, using \eqref{eq-dsha},
\[
    |u(\ga(0))-u(\ga(l_\ga))| \le \int_\ga g\, ds 
    = \int_\ga \gh\, d\sha.
\]
As $\Modp(\Ga;\Xdot,\dha,\muha) =\Modp(\Ga;X,d,\mu)=0$,
by Lemma~\ref{lem-Modp}, we have shown that $\gh$ is
a \p-weak upper gradient of 
$u$ with respect to
$(\dha,\muha)$.
The converse implication is shown similarly.
\end{proof}

\section{Newtonian spaces and capacity}
\label{sect-Newt}

Following Shanmugalingam~\cite{Sh-rev}, 
we define a version of Sobolev spaces on the metric space $X$.

\begin{deff} \label{deff-Np}
For a measurable function $u: X\to\eR$, let 
\[
        \|u\|_{\Np(X)} = \biggl( \int_X |u|^p \, d\mu 
                + \inf_g  \int_X g^p \, d\mu \biggr)^{1/p},
\]
where the infimum is taken over all upper gradients $g$ of $u$.
The \emph{Newtonian space} on $X$ is 
\[
        \Np (X) = \{u: \|u\|_{\Np(X)} <\infty \}.
\]
\end{deff}
% \medskip needed as def ends with display
\medskip

The space $\Np(X)/{\sim}$, where  $u \sim v$ if and only if 
$\|u-v\|_{\Np(X)}=0$,
is a Banach space and a lattice, see Shan\-mu\-ga\-lin\-gam~\cite{Sh-rev}.
We also define
\[
   \Dp(X)=\{u : u \text{ is measurable and  has an upper gradient
     in }   L^p(X)\}.
\]
In this paper we assume that functions in $\Np(X)$
and $\Dp(X)$
 are defined everywhere (with values in $\eR$),
not just up to an equivalence class in the corresponding function space.
This is important for upper gradients to make sense.

For a measurable set $E\subset X$, the Newtonian space $\Np(E)$ is defined by
considering $(E,d|_E,\mu|_E)$ as a metric space in its own right.
We say  that $u \in \Nploc(E)$ if
for every $x \in E$ there exists a ball $B_x\ni x$ such that
$u \in \Np(B_x \cap E)$.
The spaces $\Dp(E)$ and $\Dploc(E)$ are defined similarly.

\begin{deff} \label{deff-sobcap}
The (Sobolev) \emph{capacity} of an arbitrary set $E\subset X$ is
\[
\Cp(E) = \inf_u\|u\|_{\Np(X)}^p,
\]
where the infimum is taken over all $u \in \Np(X)$ such that
$u\geq 1$ on $E$.
\end{deff}

We say that a property holds \emph{quasieverywhere} (q.e.)\ 
if the set of points  for which the property fails
has capacity zero. 
The capacity is the correct gauge 
for distinguishing between two Newtonian functions. 
If $u \in \Np(X)$, then $u \sim v$ if and only if $u=v$ q.e.
Moreover, 
if $u,v \in \Dploc(X)$ and $u= v$ a.e., then $u=v$ q.e.

\begin{prop} \label{prop-Dp-equiv}
Let $\Om \subset X$ be open and $u: \Om \to \eR$ be measurable.
Then $u \in \Dploc(\Om,d,\mu)$ if and only if 
$u\in\Dploc(\Om,\dha,\muha)$,
and in this case
\[ 
   \gh_u(x) = g_u(x)(1+d(x,a))^2
   \quad \text{for a.e. }x \in \Om,
\] 
where $g_u$ and $\gh_u$ are the 
minimal \p-weak upper gradients with respect
to $(d,\mu)$ and 
$(\dha,\muha)$, respectively.
Moreover,
\begin{equation}  \label{eq-int-gu-guhat}
     \int_\Om g_u^p \,d\mu = \int_\Om \gh_u^p \,d\muha,
\end{equation}
and thus $\Dp(\Om,d,\mu) = \Dp(\Om,\dha,\muha)$.
\end{prop}

\begin{proof}
If $u \in \Dp(\Om,d,\mu)$ has a 
minimal \p-weak upper gradient $g_u$,
then $\gh:=g_u(x)(1+d(x,a))^2$ is a \p-weak upper gradient of $u$
with respect to $(\dha,\muha)$, by Lemma~\ref{lem-pwug}.
Moreover,
\[
   \int_\Om \gh^p \,d\muha 
   = \int_\Om g_u(x)^p (1+d(x,a))^{2p}  \frac{d\mu(x)}{(1+d(x,a))^{2p}}
    = \int_\Om g_u^p \, d\mu.
\]
Thus $u \in \Dp(\Om,\dha,\muha)$ and $\gh_u \le \gh$ a.e.
The converse inequality is shown similarly, and hence $\gh_u = \gh$ a.e.

The local case follows since
the minimal \p-weak upper gradient only
depends on the function locally, see 
\cite[Lemma~2.23 and Remark~2.28]{BBbook}.
\end{proof}

For $\Np$ we have the following corresponding result.

\begin{prop}   \label{prop-Sob-spc-equal}
Let $\Om \subset X$ be open and $u: \Om \to \eR$ be measurable.
Then the following are true\textup{:}
\begin{enumerate}
\item \label{a-Np-unbdd}
$\|u\|_{\Np(\Om,\dha,\muha)} \le \|u\|_{\Np(\Om,d,\mu)}$
and thus
$\Np(\Om,d,\mu) \subset \Np(\Om,\dha,\muha)$.
\item \label{a-Np}
If\/ $\Om$ is bounded, then $\Np(\Om,d,\mu)= \Np(\Om,\dha,\muha)$, 
as sets but with comparable norms\/ \textup{(}depending on $\Om$\textup{)}.
\item \label{a-Nploc}
$\Nploc(\Om,d,\mu)= \Nploc(\Om,\dha,\muha)$.
\end{enumerate}
\end{prop}

\begin{proof}
Clearly, 
\ref{a-Nploc} follows directly from \ref{a-Np}.
To prove \ref{a-Np}, we note that because of
Proposition~\ref{prop-Dp-equiv}, 
$u \in\Dp(\Om,d,\mu)$ if and only if 
$u \in \Dp(\Om,\dha,\muha)$, 
with equal seminorms~\eqref{eq-int-gu-guhat}.
Since $\Om$ is bounded, we also have that
$1+d(x,a)\simeq 1$ for all $x\in\Om$, which implies that 
\[
\int_\Om |u|^p\,d\mu \simeq \int_\Om |u|^p\,d\muha
\]
with comparison constants depending on $\Om$.
Thus $u \in \Np(\Om,d,\mu)$ if and only if $u \in \Np(\Om,\dha,\muha)$
with comparable norms.

Finally, \ref{a-Np-unbdd} follows immediately 
from~\eqref{eq-int-gu-guhat}
and the fact that $d\muha\le d\mu$.
\end{proof}

We shall
write $\Cphat$ and $\Cp$ for the capacities associated with the spaces 
$(\Xdot,\dha,\muha)$ and $(X,d,\mu)$, respectively.

\begin{lem} \label{lem-Cp=0-X} 
Let $E \subset X$.
Then 
$\Cphat(E)=0$ 
if and only if 
$\Cp(E)=0$.
\end{lem}

\begin{proof}
The inequality
$\Cphat(E) \simle \Cp(E)$ follows directly from
Proposition~\ref{prop-Sob-spc-equal}\,\ref{a-Np-unbdd}.

Conversely, assume that $\Cphat(E)=0$.
By Proposition~\ref{prop-Sob-spc-equal}\,\ref{a-Np},
the $\Np$-norms are comparable in $B(a,2j)$, $j \ge 1$, 
from which it follows
that $\Cp(E \cap B(a,j))=0$, see e.g.\ Lemma~2.24 in \cite{BBbook}.
The countable subadditivity of the capacity then 
shows that $\Cp(E)=0$.
\end{proof}

\section{Poincar\'e inequalities under sphericalization}
\label{sect-PI}

\emph{From now on we assume that $(X,d,\mu)$ is complete and unbounded, and
that $\mu$ is 
Ahlfors $Q$-regular with $Q>1$.
We also assume that
\begin{equation}   \label{eq-ass-pQ}
    \begin{cases} 
    1<p<Q/(2-Q), & \text{if } 1 < Q<2, \\
    p>Q/2, & \text{if } Q \ge 2,
    \end{cases}
\end{equation}
and that $(X,d,\mu)$ supports a $q$-Poincar\'e inequality,
where
\begin{equation}   \label{eq-ass-pQ-PI}
    q = \begin{cases}
      p, & \text{if } Q/2 <p \le Q, \\
      pQ/(2p-Q), & \text{if } p \ge Q.
      \end{cases}
\end{equation}}

These assumptions are satisfied e.g.\ if $X=\R^n$, $n\ge2$, equipped with
the Lebesgue measure, and $p>n/2$.
In particular, all $p>1$ are allowed in $\R^2$.

For every $Q>1$, Laakso~\cite{Laa} constructed 
a complete bounded Ahlfors $Q$-regular metric space  supporting
a $1$-Poincar\'e inequality.
Since it is bounded it does not fall within our scope here, but 
its flattening
is an unbounded complete Ahlfors $Q$-regular metric space supporting
a $1$-Poincar\'e inequality, see
Proposition~4.1 and Theorem~4.4 in 
Li--Shanmugalingam~\cite{LiShan} and Theorem~3.3 in Korte~\cite{korte07}.

Since $X$ is complete and $\mu$ is doubling and supports a 
Poincar\'e inequality, 
it follows that $(X,d)$ is \emph{quasiconvex}, i.e.\ 
there is a constant $C\geq 1$ such that each pair of points
 $x$ and $y$ in the space can be joined by a curve $\gamma$ with length
\[ 
    l_\gamma\leq Cd(x,y),
\] 
see e.g.\ Theorem~4.32 in \cite{BBbook}.
Moreover, $1<q \le Q$  and thus Theorem~3.3 in Korte~\cite{korte07}
implies that $X$ is also \emph{annularly quasiconvex}, i.e.\ 
there is a constant $\Lambda\ge 1$ such that whenever $B\subset X$ 
is a ball and $y,z\in B\setminus \tfrac12 B$, there is a curve 
$\gamma \subset \Lambda B\setminus (2\Lambda)^{-1}B$ 
connecting $y$ to $z$ and such that $l_\gamma\le \Lambda d(y,z)$.

It follows from Theorem~6.5 in 
Buckley--Herron--Xie~\cite{BuckleyHerronXie}
that also the sphericalization 
$(\Xdot,\dha)$ is both quasiconvex and annularly quasiconvex.
By  Proposition~3.1 and Theorem~3.6 in 
Li--Shanmugalingam~\cite{LiShan},
 $(\Xdot,\dha,\mu_a)$ is Ahlfors $Q$-regular and supports
a $q$-Poincar\'e inequality.

\begin{prop}    \label{prop-PI-for-muh}
The space $(\Xdot,\dha,\muha)$
supports a \p-Poincar\'e inequality and $\muha$ is doubling
on $(\Xdot,\dha)$.
\end{prop}

\begin{proof}
Since $\mu$ is Ahlfors $Q$-regular, we have by~\eqref{eq-comp-d-dhat},
\begin{align}
d\muha(x) &= \frac{d\mu(x)}{(1+d(x,a))^{2p}} 
   \simeq (1+d(x,a))^{2(Q-p)} \frac{d\mu(x)}{\mu(B(a,1+d(x,a)))^2} 
\nonumber\\
   &= (1+d(x,a))^{2(Q-p)} d\mu_a(x)
   \simeq \dha(x,\infty)^{2(p-Q)} d\mu_a(x).
   \label{eq-dha}
\end{align}
We have already observed that 
$\mu_a$
is Ahlfors $Q$-regular and supports a $q$-Poincar\'e inequality
on $(\Xdot,\dha)$.
Proposition~\ref{prop-power-Ap} with $\al=2(p-Q)$ and $\sig=Q$ now implies that
the weight $\dha(x,\infty)^{2(p-Q)}$ belongs to
\begin{itemize}
% Deleting extra vertical space, as in enumerate
\makeatletter
\setlength{\@topsep}{0pt}% Delete if extra space before list
\setlength{\itemsep}{0pt}%
\setlength{\parskip}{0pt plus 1pt}%
\makeatother
\item $A_1(\mu_a)$ when $-Q<2(p-Q)\le0$ (which is equivalent to $Q/2<p\le Q$),
\item $A_\tau(\mu_a)$ with $\tau>1$ when $0<2(p-Q)<Q(\tau-1)$, i.e.\ for $p>Q$ and $\tau>2p/Q-1$.
\end{itemize}
\unskip 
By Corollary~\ref{cor-d-al-PI}, $\muha$ is doubling on $(\Xdot,\dha)$.

Since $\mu_a$ supports a $q$-Poincar\'e inequality on $(\Xdot,\dha)$,
Corollary~\ref{cor-d-al-PI} also implies that $\muha$ supports 
a \p-Poincar\'e inequality when $Q/2<p\le Q$.

For $p>Q$, 
we first need to use Theorem~1.0.1 in Keith--Zhong~\cite{KeithZhong}
to see that $\mu_a$ supports a $q'$-Poincar\'e inequality on $(\Xdot,\dha)$
for some $q' < q=pQ/(2p-Q)$.
From this it follows, by Corollary~\ref{cor-d-al-PI}, that 
$\muha$ supports a $q'\tau$-Poincar\'e inequality on $(\Xdot,\dha)$ 
whenever $\tau >2p/Q-1$,
and thus a \p-Poincar\'e inequality, as $p > q'(2p/Q-1)$.
\end{proof}

\begin{remark}
The proof of Proposition~\ref{prop-PI-for-muh} also shows that 
if $(X,d,\mu)$ supports a $q'$-Poincar\'e inequality for some 
specific
$q'<p\le Q$, then
also $(\Xdot,\dha,\muha)$ supports a $q'$-Poincar\'e inequality. 
For $p>Q$ and $q<pQ/(2p-Q)$,
$(\Xdot,\dha,\muha)$ 
also supports a $\bar{q}$-Poincar\'e inequality for some $\bar{q}<p$
(in fact for any $\bar{q}>q'(2p-Q)/Q$, 
but not necessarily with $\bar{q}=q'$).
Such assumptions of better Poincar\'e inequalities are often used 
in the subsequent theory of \p-harmonic functions.
At the same time, since we assume that $X$ (and thus $\Xdot$)
is complete, 
the self-improvement result in 
Keith--Zhong~\cite[Theorem~1.0.1]{KeithZhong}
shows that the \p-Poincar\'e inequality implies a better 
$q'$-Poincar\'e inequality for some $q' <p$
(but with no good explicit control on $q'$).
\end{remark}

We are now ready to refine Lemma~\ref{lem-Cp=0-X}.

\begin{lem} \label{lem-Cp=0} 
Let $E \subset \Xdot$.
\begin{enumerate}
\item \label{Cp-i}
If $\infty \notin E$, then
$\Cphat(E)=0$ 
if and only if 
$\Cp(E)=0$.
\item \label{Cp-ii}
If $\infty \in E$, then
$\Cphat(E)=0$ 
if and only if 
$\Cp(E \setm \{\infty\})=0$ and $p \ge Q$.
\item \label{Cp-iii}
If $p > Q$, then
$\Cphat(E)=0$ 
if and only if $E=\emptyset$ or $E=\{\infty\}$.
\end{enumerate}
\end{lem}

\begin{proof}
\ref{Cp-i} This follows directly from Lemma~\ref{lem-Cp=0-X}.

\ref{Cp-ii} We need to determine when 
$\Cphat(\{\infty\})=0$, for which we will
use results from
Bj\"orn--Bj\"orn--Lehrb\"ack~\cite{BBLeh1}.
As $\mu_a$ is Ahlfors $Q$-regular, it is also reverse-doubling,
i.e.\ there are constants $\theta,\tau>1$ such that 
\[ 
 \mu_a(B(x,\tau r))\ge \theta \mu_a(B(x,r))
\quad \text{whenever } x\in\Xdot \text{ and }
0<r\le \diam \Xdot/2\tau.
\] 
Thus we get the following estimates,  using \eqref{eq-dha}
and denoting balls with respect to $\Xdot$ by $\Bhat$,
\begin{align*}
 \muha(\Bhat(\infty,r)) 
    & \simeq \sum_{j=0}^\infty (r\tau^{-j})^{2(p-Q)} 
       \mu_a(\Bhat(\infty,r\tau^{-j}) \setm \Bhat(\infty,r\tau^{-j-1})) \\
    & \simeq \sum_{j=0}^\infty (r\tau^{-j})^{2(p-Q)} (r\tau^{-j})^Q
     \simeq r^{2p-Q},
\end{align*}
since $2p-Q>0$, by \eqref{eq-ass-pQ}.
In the notation of \cite{BBLeh1}, this means that the dimension sets of 
$(\Xdot,\dha,\muha)$ at $\infty$ satisfy
\[ 
\lSo^\Xdot(\infty)=(0,2p-Q] \quad \text{and} \quad 
\uSo^\Xdot(\infty)=[2p-Q,\infty).
\]
Thus $\Cphat(\{\infty\})=0$ if and only if $p\le2p-Q$ (which is
equivalent to $p \ge Q$), by 
Proposition~1.3 in \cite{BBLeh1}.

\ref{Cp-iii} Let $x \in X$.
Then, following the notation in~\cite{BBLeh1}, we have 
$\uSo^X(x)=[Q,\infty)$ since $X$ is Ahlfors $Q$-regular.
It thus follows from Proposition~1.3 in \cite{BBLeh1},
or Corollary~5.39 in \cite{BBbook}, that $\CpX(\{x\})>0$.
Hence, \ref{Cp-iii} follows from \ref{Cp-i} and \ref{Cp-ii}.
\end{proof}

\section{\texorpdfstring{\p}{p}-harmonic 
functions on unbounded domains}
\label{sect-pharm}

\emph{Recall the assumptions \eqref{eq-ass-pQ} and \eqref{eq-ass-pQ-PI}
on $p$ and $X$.
From now on we also assume 
that $\Om \subset X$ is  
a nonempty open set
and we regard 
$\Om$ simultaneously as a subset 
of $(X,d,\mu)$ and 
of $(\Xdot,\dha,\muha)$.}

\medskip

In this section we apply sphericalization and the results from the earlier 
sections to obtain results about \p-harmonic functions and the Dirichlet
problem on unbounded sets. 
Let $\Lipc(\Om)$ denote the space of Lipschitz functions
with compact support in $\Om$.

\begin{deff} \label{def-quasimin}
A function $u \in \Nploc(\Om,d,\mu)$ is a
\emph{\textup{(}super\/\textup{)}minimizer} in $\Om$
(with respect to $(d,\mu)$) if 
\[ 
      \int_{\phi \ne 0} g^p_u \, d\mu
           \le \int_{\phi \ne 0} g_{u+\phi}^p \, d\mu
           \quad \text{for all (nonnegative) } \phi \in \Lipc(\Om),
\] 
where $g_u$ and $g_{u+\phi}$ are the minimal \p-weak upper gradients of $u$
and $u+\phi$ with respect to $(d,\mu)$.
(Super)minimizers with respect to $(\dha,\muha)$ are defined analogously.
A \emph{\p-harmonic function} is a continuous minimizer.
\end{deff}

For various characterizations of 
(super)minimizers  
see  \cite{BBbook} or Bj\"orn~\cite{ABkellogg}.
It was shown in Kinnunen--Shanmugalingam~\cite{KiSh01} that 
under the assumptions of doubling and a \p-Poincar\'e inequality,
a minimizer can be modified on a set of zero capacity to obtain
a \p-harmonic function. For a superminimizer $u$,  it was shown by
Kinnunen--Martio~\cite{KiMa02} that its \emph{lsc-regularization}
\[ 
 u^*(x):=\essliminf_{y\to x} u(y)= \lim_{r \to 0} \essinf_{B(x,r)} u
\] 
is also a superminimizer  and $u^*= u$ q.e. 

We are primarily interested in the Dirichlet (boundary value) problem
for \p-harmonic functions, and the associated boundary regularity. 
The most general way of treating the Dirichlet problem
is to consider Perron solutions, and in order to define
them we need superharmonic functions. 

\begin{deff}   \label{deff-superh}
A function $u:\Om \to (-\infty,\infty]$, which is not identically 
$\infty$ in any component of $\Om$, is \emph{superharmonic}
if it is lsc-regularized (i.e.\ $u=u^*$) and 
$\min\{u,k\}$ is a superminimizer for every $k \in \Z$.
\end{deff}

This is not the traditional definition of superharmonic functions,
but it is one of several equivalent characterizations
used in various places of the nonlinear literature,
cf.\ Kinnunen--Martio~\cite[Section~7]{KiMa02},
Bj\"orn~\cite[Theorem~7.1]{ABsuper}
or \cite[Theorem~9.24]{BBbook}.

Our choice of the sphericalization measure $\muha$ leads to the 
following invariance result which will be important for applications in this
section.

\begin{thm}   \label{thm-harm-equiv-X-Xdot}
A function
 $u:\Om\to\R$ is a {\rm(}super\/{\rm)}minimizer  in $\Om$ with respect to 
$(d,\mu)$
if and only if it is a {\rm(}super\/{\rm)}minimizer in $\Om$ with respect to 
$(\dha,\muha)$.

Similarly, \p-harmonicity and superharmonicity 
are the same in the two situations. 
\end{thm}

\begin{proof}
Proposition~\ref{prop-Sob-spc-equal} shows that the spaces 
$\Nploc(\Om,d,\mu)$ and  $\Nploc(\Om,\dha,\muha)$, 
appearing in Definition~\ref{def-quasimin}, coincide.
Moreover, since $\Om\subset X$,
 $\phi\in\Lip_c(\Om,d)$ if and only if $\phi\in\Lip_c(\Om,\dha)$,
i.e.\ the sets of test functions for both notions of superminimizers coincide.
Proposition~\ref{prop-Dp-equiv} implies that
\[
\int_{\phi \ne 0} g_u^p\,d\mu = \int_{\phi \ne 0} \gh_u^p\,d\muha
\quad \text{and} \quad
\int_{\phi\ne 0} g_{u+\phi}^p\,d\mu = \int_{\phi\ne 0} \gh_{u+\phi}^p\,d\muha,
\]
where $\gh_u$ and $\gh_{u+\phi}$ are the minimal \p-weak upper gradients of
$u$ and $u+\phi$ in $\Om$ with respect to $\dha$, respectively.
Taking all this into account shows 
the equivalence of the two notions of superminimizers.

It now follows directly from the definitions that also the two notions
of \p-harmonicity and superharmonicity
(with respect to $(d,\mu)$ and $(\dha,\muha)$) are equivalent.
\end{proof}

We are now ready to define the Perron solutions.
We consider the Dirichlet problem with respect to the boundary 
$\bdhat\Om$ corresponding to $\Xdot$, i.e.\ for unbounded $\Om\subset X$
we set $\bdhat\Om=\bdy\Om\cup\{\infty\}$.
This is in accordance with the definitions used in 
Heinonen--Kilpel\"ainen--Martio~\cite{HeKiMa} and
Hansevi~\cite{Hansevi2}.

\begin{deff}   \label{def-Perron}
Given 
$f : \bdhat \Om \to \eR$, let $\UU_f$ be the set of all 
superharmonic functions $u$ on $\Om$, bounded from below,  such that 
\begin{equation} \label{eq-def-Perron} 
	\liminf_{\Om \ni y \to x} u(y) \ge f(x) 
\end{equation} 
for all $x \in \bdhat \Om$.
The \emph{upper Perron solution} of $f$ is then defined to be
\[ 
    \uP f (x) = \inf_{u \in \UU_f}  u(x), \quad x \in \Om,
\]
while the \emph{lower Perron solution} of $f$ is defined by
\(
    \lP f = - \uP (-f).
\)
If $\uP f = \lP f$
and it is real-valued, then we let $Pf := \uP f$
and $f$ is said to be \emph{resolutive} with respect to $\Om$.
\end{deff}

Note that in \eqref{eq-def-Perron} the limit can be equivalently taken 
with respect to $d_a$, $\dha$ or $d$, where in the last case $y\to\infty$
is interpreted in the obvious way. 
Thus, by Theorem~\ref{thm-harm-equiv-X-Xdot}, Perron
solutions with respect to $(d,\mu)$ and $(\dha,\muha)$ are the same.

As $\Om$ is always bounded as a subset of the sphericalization $\Xdot$,
we can now use all the results about \p-harmonic functions on bounded sets for
it and they will automatically transfer to \p-harmonic functions 
and Perron solutions on 
$\Om\subset X$ even for unbounded $\Om$ (with boundary $\bdhat\Om$).

For the Perron method on $\Xdot$ we need to require that 
$\CpXdot(\Xdot \setm \Om )>0$, which by  
Lemma~\ref{lem-Cp=0} happens if and only if $\Cp(X \setm \Om)>0$ or $p< Q$.
If $p > Q$   this amounts
exactly to requiring that $\Om \ne X$, by  Lemma~\ref{lem-Cp=0}.
\begin{gather}
\text{\emph{So from now on we assume that}} \nonumber \\
\label{eq-Om-cond}
\text{\emph{$\Om\subset X$ is unbounded, and that $\Cp(X \setm \Om)>0$
or}  $p< Q$.}
\end{gather}

In this case we get, using the correspondence above, a 
rich theory also on $\Om$ seen as an unbounded open subset of 
the original space $X$.
(When $\CpXdot(\Xdot \setm \Om )=0$,
the Perron method gets somewhat pathological,
but this is not the right place to dwell upon that.)

First we observe that 
Theorem~4.1 in Bj\"orn--Bj\"orn--Shanmugalingam~\cite{BBS2}
(or \cite[Theorem~10.10]{BBbook})
shows that the
Perron solutions
are either identically $\pm \infty$ or \p-harmonic
in each component of $\Om$, and thus
in the latter case provide reasonable candidates for solutions 
of the Dirichlet problem.
Moreover,
by Theorem~7.2 in Kinnunen--Martio~\cite{KiMa02}
(or \cite[Theorem~9.39]{BBbook}), $\lP f \le \uP f$ for all 
$f: \bdhat \Om \to \eR$.

More importantly,
various resolutivity results for bounded domains
from Bj\"orn--Bj\"orn--Shan\-mu\-ga\-lin\-gam~\cite{BBS2}, \cite{BBSdir},
Hansevi~\cite{Hansevi2} and 
Bj\"orn--Bj\"orn--Sj\"odin~\cite{BBSjodin} 
transform
directly into results for unbounded $\Om$.
In unweighted and weighted $\R^n$, $n \ge 2$,
 some of these consequences recover old
results by Kilpel\"ainen~\cite{Kilp89} 
resp.\ Heinonen--Kilpel\"ainen--Martio~\cite[Chapter~9]{HeKiMa},
but we also obtain many new results.

Some of the results below were obtained by different methods in 
Hansevi~\cite{Hansevi2}
when the space $X$ is \p-parabolic (i.e.\  $p \ge Q$, see
Bj\"orn--Bj\"orn--Lehrb\"ack~\cite[Remark~8.7]{BBLeh1}), or more
generally when $\Om$ is \p-parabolic (see 
Definition~4.1 in~\cite{Hansevi2} or Definition~\ref{def-p-parabolic} below),
which is satisfied for many unbounded sets also when $p<Q$.
The Dirichlet problem on unbounded domains with respect to prime end
boundaries has 
been considered in Estep~\cite{Estep}.

Some of the obtained consequences are somewhat technical to describe,
and in order to keep the exposition limited we will not go into full generality.
To avoid misunderstanding and to make the results accessible for readers
not interested in the sphericalization $\Xdot$, we formulate them
using the capacity and other notions on $X$.
It should be fairly straightforward for the interested reader to transform
also other results from
the above mentioned papers, e.g.\ those
involving a better capacity and generalized boundaries.

\begin{thm}   \label{thm-resol-C}
Assume that \eqref{eq-ass-pQ}, \eqref{eq-ass-pQ-PI}
and \eqref{eq-Om-cond} are satisfied.
Let $f \in C(\bdhat \Om)$ and assume that $h:\bdhat \Om \to \eR$
vanishes 
$\Cp$-q.e.\ on $\bdry\Om$.
Then the following hold\/\textup{:}
\begin{enumerate}
\item If $p<Q$ and $h(\infty)=0$ then both 
$f$ and $f+h$ are resolutive and $\Hp f = \Hp (f+h)$.
\item If $p\ge Q$ then both $f$ and $f+h$ are resolutive and 
$\Hp f = \Hp (f+h)$.
Moreover, the requirement~\eqref{eq-def-Perron} in the 
definition of $\lP f$ and $\uP f$ only needs to be satisfied  at
finite boundary points $x\in\bdry\Om$.
\end{enumerate}
\end{thm}

Note that for $p>Q$ the function $h$ in Theorem~\ref{thm-resol-C}
is allowed to be nonzero only at $\infty$, since finite points have 
positive capacity.

In unweighted $\R^n$ with $ p >n/2$, $p \ne 2$,
this result (as well as the
uniqueness result in Theorem~\ref{thm-uniq-bdd-pharm} below)
is new, although the resolutivity of $f$ was shown already 
by Kilpel\"ainen~\cite[Theorem~1.10]{Kilp89}.

\begin{proof}
Resolutivity and invariance under the perturbation $h$ follow 
from \cite[Theorem~6.1]{BBS2} (or \cite[Theorem~10.22]{BBbook}) 
and the above discussion.
We also need to appeal to Lemma~\ref{lem-Cp=0},
which shows that $\{\infty\}$ has zero capacity
if and only if $p \ge Q$, and can therefore be disregarded in this case. 

To conclude the proof, let $p\ge Q$ and
$u$ be a superharmonic function on $\Om$ bounded from
below and such that~\eqref{eq-def-Perron} holds for all $x\in\bdry\Om$.
Then $u\in\UU_{f+\hh}$, where $\hh=-\infty\chi_{\{\infty\}}$, 
and hence the already proved invariance part shows that
\[
u\ge \uP(f+\hh)=\uP f.
\]
Taking infimum over all such $u$ shows that the infimum
in the definition of Perron solutions does not get smaller 
by relaxing~\eqref{eq-def-Perron}.
That it cannot get larger is trivial, since it is
taken over a larger class of functions.
\end{proof}

The following theorem provides us with a unique solution of the
Dirichlet problem on unbounded domains. 
Note, however, that the point at infinity is regarded as a boundary point 
even if the usual boundary $\bdry\Om$ is bounded.
This additional requirement is necessary if $p < Q$.

\begin{thm} \label{thm-uniq-bdd-pharm}
Assume that \eqref{eq-ass-pQ}, \eqref{eq-ass-pQ-PI}
and \eqref{eq-Om-cond} are satisfied.
Let $f \in C(\bdhat\Om)$.
Then $u=Pf$ is the 
unique bounded \p-harmonic function $u$ on $\Om$
such that
\[ 
     \lim_{\Om \ni y \to x} u(y)=f(x)
     \quad \text{for $\Cp$-q.e.\ } x \in \bdry \Om
\] 
and also for $x=\infty$ when $p<Q$.
\end{thm}

\begin{proof}
This follows directly from \cite[Theorem~10.24]{BBbook}, together with
Lemma~\ref{lem-Cp=0} and the above discussion.
\end{proof}

For Newtonian functions, and more generally 
Dirichlet functions, we obtain the following resolutivity and
uniqueness results
corresponding to Theorems~\ref{thm-resol-C} and~\ref{thm-uniq-bdd-pharm}.

\begin{thm} \label{thm-Np-res}
Assume that \eqref{eq-ass-pQ}, \eqref{eq-ass-pQ-PI}
and \eqref{eq-Om-cond} are satisfied.
Let $f\in\Dp(X,d,\mu)$ and assume that $h:\bdhat\Om \to \eR$ 
vanishes $\Cp$-q.e.\ on $\bdry\Om$.
Then the following hold\/\textup{:}
\begin{enumerate}
\item
If $p<Q$, $h(\infty)=0$ and $\lim_{y\to\infty}f(y)=:f(\infty)$ 
exists\/ \textup{(}in $\eR$\textup{)},
then both $f$ and $f+h$ are resolutive and $\Hp f = \Hp (f+h)$.
\item  \label{b-b}
If $p\ge Q$, then both $f$ and $f+h$ are resolutive and $\Hp f = \Hp (f+h)$.
Moreover, 
the requirement~\eqref{eq-def-Perron} in the 
definition of $\lP f$ and $\uP f$ 
only needs to be satisfied  at
finite boundary points $x\in\bdry\Om$.
\end{enumerate}
\end{thm}

\begin{proof}
The resolutivity and invariance follow from the sphericalization
process together with
(the bounded case of) Theorem~7.6 in Hansevi~\cite{Hansevi2},
if we can show that $f\in\Dp(\Xdot,\dha,\muha)$.
To this end, note that $f\in\Dp(X,\dha,\muha)$, by 
Proposition~\ref{prop-Dp-equiv}.
Let $\gh\in L^p(X,\muha)$ be an upper gradient of $f$ in $X$ 
with respect to $\dha$, and let $\gh(\infty)$ be arbitrary.

If $p\ge Q$ then $\Cphat(\{\infty\})=0$, by Lemma~\ref{lem-Cp=0},
and hence \p-almost every curve 
in $\Xdot$ avoids $\infty$
(by \cite[Proposition~1.48]{BBbook}),
which immediately implies that $\gh$ is a \p-weak upper gradient of $f$
in $\Xdot$, and thus $f\in\Dp(\Xdot,\dha,\muha)$.

For $p<Q$, let $\ga\subset\Xdot$ be a rectifiable curve.
If $\ga \subset X$, there is nothing to prove. 
So, by
splitting $\ga$ into parts and reversing the orientation, if necessary,
we can assume that $\ga^{-1}(\{\infty\})=\{0\}$.
The continuity of $f$ at $\infty$ then yields
\[
|f(\ga(0))-f(\ga(l_\ga))| = \lim_{t\to0} |f(\ga(t))-f(\ga(l_\ga))|
\le \int_\ga \gh\,ds.
\]
Since $\ga$ was arbitrary, we conclude that $\gh$ is an upper gradient
of $f$ in $\Xdot$, and thus $f\in\Dp(\Xdot,\dha,\muha)$.

The last part in \ref{b-b}
is proved in the same way as the similar statement in 
Theorem~\ref{thm-resol-C}.
\end{proof}

\begin{prop} 
Assume that \eqref{eq-ass-pQ}, \eqref{eq-ass-pQ-PI}
and \eqref{eq-Om-cond} are satisfied.
Let $f\in\Dp(X,d,\mu)$ be bounded and assume that $u$ is a bounded
\p-harmonic function in $\Om$ such that
\begin{equation}   \label{eq-lim-qe}
     \lim_{\Om \ni y \to x} u(y)=f(x)
     \quad \text{for $\Cp$-q.e.\ } x \in \bdry \Om.
\end{equation}
If $p<Q$, assume in addition that 
$\lim_{\Om \ni y \to \infty} u(y)=\lim_{y\to\infty}f(y)$. 
Then $u=Pf$.
\end{prop}

Note that, unlike in Theorem~\ref{thm-uniq-bdd-pharm},
the existence of a function satisfying~\eqref{eq-lim-qe}
for noncontinuous boundary data is not guaranteed by the Kellogg property
(Theorem~\ref{thm-Kellogg} below).

\begin{proof}
By the proof of Theorem~\ref{thm-Np-res},
$f \in \Dp(\Xdot,\dha,\muha)$ and hence $f\in\Np(\Xdot,\dha,\muha)$
since it is bounded.
Thus, the statement
follows directly from \cite[Corollary~10.16]{BBbook},
together with the sphericalization process and
Lemma~\ref{lem-Cp=0}.
\end{proof}

Also boundary regularity results carry over to unbounded domains, 
the most important is maybe the Kellogg property,
which we obtain using \cite[Theorem~3.9]{BBS}, together with
Lemma~\ref{lem-Cp=0}.
Recall that $x\in \bdhat\Om$ is called \emph{regular} if 
\begin{equation}    \label{eq-def-regular-pt}
\lim_{\Om \ni y \to x} Pf(y)=f(x)
\quad \text{for all }f\in C(\bdhat\Om).
\end{equation}

\begin{thm} \label{thm-Kellogg} \textup{(Kellogg property)}
Assume that \eqref{eq-ass-pQ}, \eqref{eq-ass-pQ-PI}
and \eqref{eq-Om-cond} are satisfied.
The set of irregular boundary points in $\bdy\Om$
has $\Cp$-capacity zero.
Moreover, $\infty$ is always regular if $p<Q$.
\end{thm}

Useful properties of boundary regularity are its locality and 
the barrier characterization, which transfer to unbounded domains in 
the following way.
A superharmonic function $u$ in $\Om$ is a \emph{barrier} 
at $x_0\in\bdhat\Om$ if
\[
\lim_{\Om\ni y\to x_0} u(y)=0 \quad \text{and} \quad
\liminf_{\Om\ni y\to x} u(y)>0 \quad \text{for every } x\in\bdhat\Om\setm\{x_0\}.
\]

\begin{thm}  \label{thm-barrier+local}
Assume that \eqref{eq-ass-pQ}, \eqref{eq-ass-pQ-PI}
and \eqref{eq-Om-cond} are satisfied.
A point $x_0\in\bdhat\Om$ is regular if and only if there exists a 
barrier at $x_0$.
\textup{(}Equivalently, the barrier can be chosen positive and 
continuous.\textup{)}
In this case, \eqref{eq-def-regular-pt} holds for all bounded
$f:\bdhat\Om\to\R$ which are continuous at $x_0$.

Moreover, regularity is local in the following sense\/\textup{:}
\begin{enumerate}
\item \label{a-aa}
A finite boundary point $x_0\in\bdy\Om$ is regular with respect 
to $\Om$ if and only if it is regular with respect to $\Om\cap G$ for
some {\rm(}or equivalently all\/{\rm)} open $G\ni x_0$. 
\item  \label{b-bb}
The point $\infty\in\bdhat\Om$ is regular 
with respect to $\Om$ if and only if it is regular with respect to
$\Om\setm K$ for some {\rm(}or equivalently all\/{\rm)} compact $K$.
\end{enumerate}
\end{thm}

\begin{proof}
This follows directly from the sphericalization process, together with
Theorems~4.2 and~6.1 in Bj\"orn--Bj\"orn~\cite{BB}
(or \cite[Theorem~11.11]{BBbook}).
\end{proof}

Some of the above results  are new also for unweighted 
$\R^n$, $n \ge 2$, with
$p > n/2$, 
but the Kellogg property and the barrier
characterization appeared already in
Kilpel\"ainen~\cite[Theorems~1.5 and Corollary~5.6]{Kilp89}
in this setting.
We also obtain several
new characterizations of boundary regularity in 
unbounded sets corresponding to the results in 
Bj\"orn--Bj\"orn~\cite[Theorem~6.1]{BB} (or \cite[Theorem~11.11]{BBbook}),
see also Heinonen--Kilpel\"ainen--Martio~\cite[Chapter~9]{HeKiMa}.

A direct consequence of 
Theorem~\ref{thm-barrier+local}\,\ref{a-aa} is that the Wiener type criterion
from Bj\"orn--MacManus--Shanmugalingam~\cite{BMS} 
and Bj\"orn~\cite{JB-pfine}, \cite{JBCalcVar} can be applied also to
finite boundary points in unbounded domains.
Regularity of the point at $\infty$ will be discussed in the next section.

Other boundary regularity results that generalize from bounded
to unbounded sets are the trichotomy classification
into regular, semiregular and strongly irregular boundary points
from Bj\"orn~\cite{ABclass} (or \cite[Chapter~13]{BBbook}).
These results can be applied to finite boundary points 
as well as to $\infty$.
Moreover, the results on approximation by regular sets
and on so-called Wiener solutions of the Dirichlet problem
from Bj\"orn--Bj\"orn~\cite{BB2} (or \cite[Chapter~14]{BBbook})
generalize in a similar way.

Furthermore, since the \p-energy is preserved under sphericalization, 
also quasiminimizers
are preserved in just the same way as \p-harmonic functions.
We can thus
generalize many earlier boundary regularity results for quasiminimizers
from bounded to unbounded sets (provided that $X$ satisfies our
standing assumptions).
These include results in
\cite[Sections~4 and~5]{ABkellogg}, 
\cite[Section~6]{ABclass}, \cite[Section~7]{ABcluster}, 
\cite[Theorem~1.1]{BB2},
\cite[Section~6]{BMartio},
\cite[Theorems~2.12 and~2.13]{Bj02} 
and \cite[Theorem~1.1]{JBCalcVar}.

\section{Resolutivity and regularity at \texorpdfstring{$\infty$}{oo}}
\label{sect-infty}

For unbounded $\Om$, what happens at $\infty$ is of particular interest.
Recall that by Theorem~\ref{thm-Kellogg}, $\infty$ is always regular
if $p<Q$.
Theorem~\ref{thm-barrier+local}\,\ref{b-bb}, combined 
with Theorem~7.5 in Bj\"orn~\cite{ABcluster} 
(or \cite[Theorem~11.27]{BBbook}), immediately
implies the following result.

\begin{prop}  \label{prop-infty-components}
Assume that \eqref{eq-ass-pQ}, \eqref{eq-ass-pQ-PI}
and \eqref{eq-Om-cond} are satisfied.
Let $k>0$. Then $\infty\in\bdhat\Om$ is regular with respect to $\Om$ 
if and only if it is regular with respect to every unbounded 
component $G$ of $\Om\setm \itoverline{B(a,k)}$.
This in particular guarantees regularity of the point at $\infty$ 
if there are no such unbounded  
components.
\end{prop}

A simple application of the last part of 
Proposition~\ref{prop-infty-components} is demonstrated in the following
example.

\begin{example}  
The point $\infty\in\bdy\Om$ is regular with respect to
\[
\Om= (0,1)^2 \cup \bigcup_{j=1}^\infty (2^{-j},2^{1-j}) \times (0,2^{j}),
\]
since $\Om\setm \itoverline{B(0,2)}$ consists only of bounded components.
\end{example}

The capacity $\Cphat$ (or rather its variational 
analogue $\cphat$) near
$\infty$ plays an important role here through Wiener type criteria.
Using~\eqref{eq-int-gu-guhat}, $\cphat$ can be described by means of functions
on $X$. This
makes it possible to rewrite the Wiener type integrals at $\infty$,
appearing in Bj\"orn~\cite{JB-pfine}, \cite{JBCalcVar} and 
Bj\"orn--MacManus--Shan\-mu\-ga\-lin\-gam~\cite{BMS}, in terms
of $\Om$ and $X$.
For a more hands-on result we use the fact that porosity is sufficient
for boundary regularity and formulate the following practical condition.

\begin{thm}   \label{thm-infty-porous}
Assume that \eqref{eq-ass-pQ}, \eqref{eq-ass-pQ-PI}
and \eqref{eq-Om-cond} are satisfied.
Assume that for some $k>0$ and for each unbounded
component $G$ of $\Om\setm \itoverline{B(a,k)}$ 
there exist $\theta>0$ and $x_j\in X$ such that
$d(x_j,a)\to\infty$ as $j\to\infty$ and
\begin{equation}   \label{eq-porous-infty}
B(x_j,\theta d(x_j,a)) \cap G = \emptyset
\quad \text{for all } j=1,2,\ldots.
\end{equation}
Then $\infty\in\bdhat\Om$ is regular.
\end{thm}

\begin{proof}
Simple geometrical considerations show that \eqref{eq-porous-infty}
implies the existence of $\theta'>0$ so that for sufficiently large $j$
(and with $\Bhat$ denoting balls with respect to $\Xdot$),
\[
\Bhat (x_j,\theta'\dha(x_j,\infty)) 
        \subset \Bhat (\infty, 2\dha(x_j,\infty)) \setm G,
\]
i.e.\ that $G$ is \emph{porous} at $\infty$ with respect to $\dha$.
The sphericalization argument together with
\cite[Corollary~11.25\,(c)]{BBbook}
then implies that $\infty\in\bdhat G$ is 
regular with respect to the component $G$.
Since this is true for every unbounded component of 
$\Om\setm \itoverline{B(a,k)}$,
Proposition~\ref{prop-infty-components} concludes the proof.
\end{proof}

\begin{remark}   \label{rem-one-porous}
The proof of Theorem~\ref{thm-infty-porous} shows that if 
\eqref{eq-porous-infty} holds for one particular unbounded component 
of $\Om\setm \itoverline{B(a,k)}$ then 
$\infty$
is regular
with respect to that component.
\end{remark}

In order to capture the behaviour of functions from different
directions at $\infty$,
we will consider the Mazurkiewicz metric $\dha_M$ on $\Om$, generated by $\dha$,
cf.\ Bj\"orn--Bj\"orn--Shanmugalingam~\cite{BBSdir}, \cite{BBSfincon}.
For simplicity we will restrict ourselves to the case when
$\Om$ is connected and satisfies \eqref{eq-Om-cond}.
For $x,y\in\Om$, let
\[
\dha_M(x,y) = \inf_E \diam_{\dha} E,
\]
where the infimum is taken over all connected sets $E\subset \Om$ containing
both $x$ and $y$, and the diameter is taken with respect to $\dha$.
It also gives rise to the Mazurkiewicz boundary $\bdhat_M\Om$ 
with respect to $\dha_M$ in the usual way through completion.
The Mazurkiewicz metric $d_M$ on $\Om$ generated by $d$, and the
corresponding boundary $\bdy_M\Om$, 
are defined similarly.

It was shown in \cite{BBSdir} that upper gradients, Newtonian spaces and
\p-(super)\-har\-mon\-ic functions within $\Om$
are the same with respect to the Mazurkiewicz metric and the 
original metric generating it.
The only change needed in the definition of Perron solutions with 
respect to the Mazurkiewicz boundary is that the $\liminf$ in 
\eqref{eq-def-Perron} is with respect to $d_M$ and is required on
$\bdhat_M\Om$.
To be able to use the results from \cite{BBSdir} we assume that $\Om$ is
\emph{finitely connected at the boundary} with respect to $\dha$, 
cf.\ N\"akki~\cite{nakki70}, V\"ais\"al\"a~\cite{vaisala}, 
as well as  \cite{BBSdir}
and \cite{BBSfincon}.
This is equivalent to requiring the following two conditions:
\begin{enumerate}
\renewcommand{\theenumi}{\textup{(\roman{enumi})}}%
\renewcommand{\labelenumi}{\theenumi}
\item \label{it-fin-conn-fin}
$\Om$ is finitely connected at every finite $x\in\bdry\Om$, 
i.e.\ for every $r>0$ and every $x\in\bdy\Om$ there is an open 
neighbourhood $U\subset B(x,r)$ of $x$ such that $\Om\cap U$  
consists of only finitely many components; 
\item \label{it-fin-conn-infty}
for every $k>0$ there is a compact set $K \supset \itoverline{B(a,k)}$
such that $\Om \setm K$ has only finitely many components.
\end{enumerate}

By Proposition~2.5 in~\cite{BBSfincon},
the conditions 
\ref{it-fin-conn-fin} and 
\ref{it-fin-conn-infty} 
can equivalently be stated as follows.
For $x \in \bdy \Om$ and $r>0$, let
$N(r,x)$ be the number of components $V$
of $B(x,r) \cap \Om$ such that $x \in \overline{V}$, and let $H(r,x)$ be the union
of all the \emph{other} components of $B(x,r) \cap \Om$.
Similarly, let $N(r,\infty)$ be the number of unbounded components 
of $\Om \setm \itoverline{B(a,r)}$, and let $H(r,\infty)$ be the union
of all the bounded components of $\Om \setm \itoverline{B(a,r)}$.
Then 
\ref{it-fin-conn-fin} and 
\ref{it-fin-conn-infty} 
are equivalent to the following two conditions, respectively:
\begin{enumerate}
\renewcommand{\theenumi}{\textup{(\roman{enumi}$'$)}}%
\renewcommand{\labelenumi}{\theenumi}
\item \label{it-fin-conn-fin-var}
$N(r,x)< \infty$ and $x \notin \itoverline{H(r,x)}$,
for every $x \in \bdy \Om$ and $0<r<1$;
\item \label{it-fin-conn-infty-var}
$N(r,\infty)< \infty$ and $H(r,\infty)$ is bounded 
for every $r>1$.
\end{enumerate}

The condition~\ref{it-fin-conn-infty}, 
or equivalently \ref{it-fin-conn-infty-var},
means that the Mazurkiewicz metric $\dha_M$
distinguishes between different copies of $\infty$, each
corresponding to a decreasing sequence of unbounded components
of $\Om\setm \itoverline{B(a,k)}$, $k=1,2,\ldots$\,.
For unbounded
$\Om\subset\R^n$ with sufficiently smooth boundary it is only this
requirement at $\infty$ that takes effect since finite
connectedness at finite boundary points is 
automatically satisfied for such smooth domains.

When discussing the Dirichlet problem with respect to the Mazurkiewicz
boundary $\bdhat_M\Om$ we will restrict ourselves to 
$f\in C(\bdhat_M\Om)$, which is in fact equivalent to 
$f\in C(\bdy_M\Om)$ together with the requirement that 
\begin{equation}    \label{eq-lim-along-Om-k}
\lim_{\bdy_M\Om\ni x\to\infty} f(x) \quad
\text{exists and is finite
\emph{along each decreasing sequence} } \{\Om_k\}_{k=1}^\infty
\end{equation} 
of unbounded components $\Om_k\subset \Om\setm \itoverline{B(a,k)}$, 
$k=1,2,\ldots$\,.
Here we say that $X\ni x_n\to\infty$, as $n\to\infty$, 
along such a sequence 
$\Om_1 \supset \Om_2 \supset \ldots$
if for every $k$ there exists $N_k$
such that all $x_n$ with $n\ge N_k$ belong to the $d_M$-closure of $\Om_k$.
Note that  the limit is allowed to be different for different
directions towards $\infty$, i.e.\ 
for different sequences $\{\Om_k\}_{k=1}^\infty$.

Under the assumption of finite connectedness at the boundary,
it can be verified that
$\bdhat_M\Om$ equals $\bdry_M\Om$, together with all the copies
of $\infty$ from different directions.
Finiteness at the boundary is equivalent to the
compactness of the $\dha_M$-closure of $\Om$,
see Bj\"orn--Bj\"orn--Shanmugalingam~\cite[Theorem~1.1]{BBSfincon} or 
Karmazin~\cite[Theorem~1.3.8]{karmazin2008}.
In the terminology of 
Adamowicz--Bj\"orn--Bj\"orn--Shanmugalingam~\cite{prime} and
Estep~\cite{Estep}, the sequence
$\{\Om_k\}_{k=1}^\infty$ can be identified with a so-called \emph{prime end} 
at $\infty$, see~\cite[Theorem~9.6 and Corollary~10.9]{prime}.

The following example shows that there can be uncountably many such
directional sequences towards $\infty$.

\begin{example} \label{ex-uncountable}
Let 
\[
A=\{\al=0.\al_1\ldots \al_n \in (0,1): 
         \al_j\in\{0,1\}, \ n=1,2,\ldots\}
\] 
be the set of all $\al\in(0,1)$ with finite binary 
expansions.
Let $\|\al\|$ denote the last nonzero position.
For each $\al\in A$, consider the ray
\[
F_a= \{z=r e^{i\al\pi}\in\C: r \ge\|\al\|\}
\]
and let $\Om$ be the upper half-plane with all these rays $F_\al$, 
$\al\in A$, removed.
Then $\Om$ is finitely connected at the boundary and there are uncountably 
many directions towards $\infty$ within $\Om$, corresponding to each 
$\al \in (0,1) \setm A$.
Theorem~\ref{thm-infty-porous} implies that the point at $\infty$ is 
regular with respect to $\Om$.
\end{example}

The following result is  a direct consequence of 
the sphericalization process and Theorem~8.2 in 
Bj\"orn--Bj\"orn--Shanmugalingam~\cite{BBSdir}.

\begin{thm}  \label{thm-resol-a-b}
Assume that $\Om$ satisfies \ref{it-fin-conn-fin} and \ref{it-fin-conn-infty},
and that \eqref{eq-ass-pQ}, \eqref{eq-ass-pQ-PI}
and \eqref{eq-Om-cond} are satisfied.
Let $f\in C(\bdy_M\Om)$ be such that
$\lim_{\bdy_M\Om\ni x\to\infty}f(x)$ exists and is finite 
along each decreasing sequence of unbounded components of 
$\Om\setm \itoverline{B(a,k)}$, $k=1,2,\ldots$\,,
in the sense of~\eqref{eq-lim-along-Om-k}.
Then $f$ is resolutive with respect to the Mazurkiewicz boundary.
\end{thm}

Here, and in Theorem~\ref{thm-reg-Om-k}, it is assumed that
$f$ takes the value at a boundary point at infinity given by the above 
limit, which may depend on the direction towards $\infty$.

Boundary regularity with respect to the Mazurkiewicz boundary in
bounded domains that
are finitely connected at the boundary was studied in
Bj\"orn~\cite{ABgenreg}.
The results therein can therefore be reformulated using sphericalization
for unbounded domains as well.
We leave the details to the interested reader and restrict ourselves
to the following
special case which can be combined with the conditions in
Proposition~\ref{prop-infty-components}, Theorem~\ref{thm-infty-porous}
and Remark~\ref{rem-one-porous}.

\begin{thm}  \label{thm-reg-Om-k}
Let $\Om$ and $f$ be as in Theorem~\ref{thm-resol-a-b}.
Assume that 
\begin{equation}    \label{eq-lim=A}
\lim_{\bdy_M\Om\ni x\to\infty} f(x) = A \in\R  \quad
\text{along a decreasing sequence }  \{\Om_k\}_{k=1}^\infty
\end{equation}
of unbounded components $\Om_k\subset \Om\setm \itoverline{B(a,k)}$,
$k=1,2,\ldots$\,.
If $\infty$ is regular with respect to $\Om_j$
for some $j$,
then the Perron solution $P^M_\Om f$ 
in $\Om$ with respect to the Mazurkiewicz boundary satisfies
\[
\lim_{\Om\ni x\to\infty} P^M_\Om f(x) = A  \quad
\text{along }  \{\Om_k\}_{k=1}^\infty,
\] 
i.e.\ the point at infinity in the direction of $\{\Om_k\}_{k=1}^\infty$
is regular.
\end{thm}

\begin{proof}
Let $\eps>0$.
It is easily verified using \eqref{eq-lim=A} that there exists $k \ge j$ 
such that $|f-A|<\eps$ on $\bdy_M\Om_k \cap \bdy_M\Om$.
By Theorem~\ref{thm-resol-a-b}, $f$ is resolutive with respect to the
Mazurkiewicz boundary.
It then follows 
from the definition of Perron solutions with respect
to $\bdy_M\Om$ and $\bdy_M\Om_k$ that
\[
  P^M_\Om f = P^M_{\Om_k} f_k \quad \text{in $\Om_k$,} \quad \text{where }
f_k= \begin{cases}
    f & \text{on } \bdy_M\Om_k \cap \bdy_M\Om, \\
    P^M_\Om f & \text{on } \bdy_M\Om_k \cap \Om,  \end{cases}
\]
cf.\ Lemma~3.3 in Bj\"orn~\cite{JBMatsue}.
Let
\[
\tf_k= \begin{cases}
    A+\eps & \text{on } \bdy\Om_k \cap \bdy\Om, \\
    P^M_\Om f & \text{on } \bdy\Om_k \cap \Om.
\end{cases}
\]
Then it is easy to see that 
$P^M_{\Om_k} f_k \le \uP_{\Om_k} \tf_k$ in $\Om_k$.
By Corollary~4.4 in Bj\"orn--Bj\"orn~\cite{BB}
(or \cite[Corollary~11.3]{BBbook}),
$\infty$ is regular with respect to $\Om_k \subset \Om_j$.
Since $\bdy \Om_k \cap \Om \subset \bdy B(a,k)$,
we see that
$\tf_k$ is continuous at $\infty\in\bdhat\Om_k$, and thus we obtain
using \cite[Theorem~6.1]{BB}
(or \cite[Theorem~11.11]{BBbook}) that 
\[
\limsup_{\Om\ni x\to\infty} P^M_\Om f(x) 
\le \limsup_{\Om\ni x\to\infty} \uP_{\Om_k} \tf_k(x) \le A+\eps
\quad \text{along }  \{\Om_k\}_{k=1}^\infty.
\]
Letting $\eps\to0$ and applying the same argument to $-f$ concludes the proof.
\end{proof}

\begin{example}  \label{ex-fingers}
For $j=0,1,2,\ldots$\,, let 
\[
F_j = \{x=(x_1,x_2)\in\R^2:  x_2=2^{j}x_1 \ge 2^j\}
\]
and $\Om = (0,\infty)^2\setm \bigcup_{j=0}^\infty F_j$.
Then $\Om$ is finitely connected at the boundary and
the ``fingers'' of $\Om$ (each between $F_j$ and $F_{j+1}$)
determine countably many directions towards $\infty$, accumulating
towards the strip $(0,1)\times(0,\infty)\subset\R^2$,
which also determines one direction towards $\infty$.
By Theorem~\ref{thm-infty-porous},  
$\infty\in\bdhat\Om$ is regular with respect to $\Om$, 
and so are all the $\bdhat_M\Om$-boundary
points at infinity, by Theorem~\ref{thm-reg-Om-k}.
\end{example}

\begin{example}  \label{ex-fingers-prime}
For $j=0,1,2,\ldots$\,, let 
\[
F'_j = \{x=(x_1,x_2)\in\R^2:  x_2=2^{j}x_1 \ge 1\}
\]
and $\Om' = (0,\infty)^2\setm \bigcup_{j=0}^\infty F'_j$.
The ``fingers'' of $\Om$ (each between $F'_j$ and $F'_{j+1}$)
determine countably many directions towards $\infty$, accumulating
towards the positive $x_2$-axis.
This sequence of ``fingers'' also determines one direction 
towards $\infty$ even though there is no single finger corresponding to it.

Since
 $\Om$ is \emph{not} finitely connected at the boundary,
Theorems~\ref{thm-resol-a-b}, \ref{thm-reg-Om-k} and
\ref{thm-mazurk-negligible}
are \emph{not} applicable.
Nevertheless, Theorem~\ref{thm-infty-porous} shows that 
the point at $\infty$ is regular with respect to $\Om$.
\end{example}

The influence of each of the directions to infinity on the Dirichlet
problem is determined by the capacity $\Cphat^M$, which is adapted to
$\Om$ and the Mazurkiewicz metric, as in 
Bj\"orn--Bj\"orn--Shanmugalingam~\cite{BBSdir}.
To keep the exposition simple, we will 
restrict ourselves to the following sufficient
condition guaranteeing that a point at infinity is negligible along 
a decreasing sequence 
$\Om_1\supset\Om_2\supset\ldots$
of unbounded components of $\Om\setm \itoverline{B(a,k)}$,
$k=1,2,\ldots$\,,
cf.\ Definition~4.1 in Hansevi~\cite{Hansevi2}.

\begin{deff} \label{def-p-parabolic}
We say that a decreasing sequence $\{\Om_k\}_{k=1}^\infty$
of unbounded 
components of $\Om\setm \itoverline{B(a,k)}$, $k=1,2,\ldots$\,,
is \p-\emph{parabolic towards $\infty$} if there 
exist $u_j\in\Np(\Om)$ satisfying $u_j=0$ in $\Om\cap B(a,j)$,
\begin{equation}  \label{eq-p-parab}
\int_{\Om_j} g_{u_j}^p\,d\mu \to 0, \quad \text{as } j\to\infty,
\end{equation}
and $\liminf_{x\to\infty} u_j(x)\ge1$ along $\{\Om_k\}_{k=1}^\infty$
for each $j=1,2,\ldots$\,.
\end{deff}

\begin{thm} \label{thm-mazurk-negligible}
Let $\Om$ and $f$ be as in Theorem~\ref{thm-resol-a-b}.
If $\{\Om_k\}_{k=1}^\infty$ is \p-parabolic towards $\infty$
then the point at $\infty$ is negligible for 
the Perron solution $Pf$ along $\{\Om_k\}_{k=1}^\infty$, i.e.\ 
the requirement   
$\liminf_{\Om\ni x\to\infty} u(x)\ge f(\infty)$ in the definition of 
$\uP f$
does not need to be satisfied when $x\to\infty$ along $\{\Om_k\}_{k=1}^\infty$.
\end{thm}

\begin{proof}  
This follows from Theorem~8.2 in 
Bj\"orn--Bj\"orn--Shan\-mu\-ga\-lin\-gam~\cite{BBSdir} and the fact that 
\eqref{eq-p-parab}, together with~\eqref{eq-int-gu-guhat},
implies that the corresponding point in
the Mazurkiewicz boundary $\bdhat_M\Om$ has zero $\Cphat^M$-capacity. 
See the proof of Theorem~\ref{thm-resol-C} for further details.
\end{proof}

Note that if, as in Example~\ref{ex-uncountable}, there are uncountably many
directions towards $\infty$, it may happen
that $\Cphat^M(E,\Om)>0$,
even if each direction towards $\infty$  is \p-parabolic,
where $E$ is the set consisting of all the \p-parabolic
directions towards $\infty$.
Thus we cannot
conclude that all of $E$ can be ignored in the definition of Perron solutions, 
at least not using the  technique here.

\section{\texorpdfstring{\p}{p}-harmonic measure 
is nonadditive on null sets}
\label{sect-pharm-meas}

The \emph{\p-harmonic measure} of a set $E \subset \bdy \Om$
is $\pharm(E;\Om):=\uHp \chi_E$, where $\chi_E$ is the characteristic
function of $E$.
When $p =2$ it becomes the usual (upper) harmonic measure.

For the upper half plane $\R^2_\limplus=\{(x,y) \in \R^2: y > 0\}$,
equipped with the Lebesgue measure $m$,
Llorente--Manfredi--Wu~\cite{LoMaWu}
showed that for any $p \ne 2$ there are
sets $E_1,\ldots,E_k$ 
such that $\bigcup_{j=1}^k E_j=\R$, but
$\pharm(E_j;\R^2_\limplus)=m(\R \setm E_j)=0$, $j=1,\ldots,k$.
Since $\pharm(\R;\R^2_\limplus)=1$, it follows that
there are two sets $A_1,A_2\subset \R$ such that
\[
   \pharm(A_1;\R^2_\limplus)= \pharm(A_2;\R^2_\limplus) =0 
  < \pharm(A_1 \cup A_2;\R^2_\limplus),
\]
showing that the \p-harmonic measure is not finitely subadditive on zero sets.
As in Definition~\ref{def-Perron} and \cite[Section~11]{HeKiMa}, 
the \p-harmonic measure in~\cite{LoMaWu} 
is taken with respect to the compactified boundary 
$\bdhat \R^2_\limplus=\R\cup\{\infty\}$.
The equality 
$\pharm(\R;\R^2_\limplus) = \pharm(\bdhat \R^2_\limplus;\R^2_\limplus) = 1$ 
then follows
from \cite[Theorem~11.4]{HeKiMa} and the following lemma.
The lemma
was mentioned in \cite{LoMaWu}, but 
for the reader's convenience we provide a proof.

\begin{lem}  \label{lem-om-infty-0}
Let $X=\R^n$, $n \ge 2$,  equipped with the Lebesgue measure, 
and let $p>1$.
Then $\pharm(\{\infty\};\Rnp)=0$.
\end{lem}

\begin{proof}
Let $f=\chi_{\{\infty\}}$.
Consider first the case when $1<p<n$.
Let 
\[
   u_k(x)=1- \biggl(\frac{|x-(0,\ldots,0,-k)|}{k}\biggr)^{(p-n)/(p-1)}, 
   \quad x \in \Rnp.
\]
Then $u_k \in \UU_{f}$ and thus
for each $x=(x_1,\ldots,x_n) \in \Rnp$,
\[
   \pharm(\{\infty\};\Rnp)(x)
   \le u_k(x)
   \le 1-\biggl| \frac{x_n+k}{k}\biggr|^{(p-n)/(p-1)}
   \to 0, 
   \quad \text{as } k \to \infty.
\]
For $p\ge n$, $k \ge 2$ and $x \in \Rnp$, we instead let
\[
   u_k(x)=\begin{cases}
     \displaystyle \frac{1}{k}|x|^{(p-n)/(p-1)}, & \text{if } p >n, \\[2mm]
\displaystyle \log\frac{|x-(0,\ldots,0,-k)|}{k}, & \text{if } p =n.
    \end{cases}
\]
These are estimated similarly.
\end{proof}

The proof in~\cite{LoMaWu}
uses an idea of Wolff~\cite{wolff} which involves intricate use of scaling
and translation invariance, and is thus not applicable to bounded domains.
Now we are able to construct similar bounded examples using
sphericalization.
As mentioned in \cite{LoMaWu}, 
by adding dummy variables one directly obtains similar 
examples also on $\Rnp$, $n \ge 3$.
Using the sphericalization technique in this paper we
obtain the following example on the sphere $\Sp^{n}$.

\begin{example} \label{ex-Sn}
Let $X=\R^n$ and $\Xdot=\Sp^n$ be its sphericalization, 
$n \ge 2$, $1<p< \infty$, $p \ne 2$,
$\Om = \Rnp$, and $A_1$ and $A_2$ be as above.
When $\Om$ is seen as a bounded subdomain of $\Xdot$, we have
that 
\[
   \pharm(A_1;\Om)= \pharm(A_2;\Om) =0 
  < \pharm(A_1 \cup A_2;\Om),
\]
showing that also in this situation
the \p-harmonic measure is not finitely subadditive on zero sets.
The sets $E_j$ transfer similarly.

Note that the metric $\dha$  from \eqref{eq-metric-dh},  that
we equip the sphere $\Xdot = \Sp^{n}$ with, is \emph{not}
the usual spherical (inner) metric, nor the metric
induced by $\R^{n+1}$.
Moreover, $\Xdot = \Sp^{n}$ is
equipped with a measure which depends on $p$.
\end{example}

Just as sphericalization can be used to map an unbounded
space into a bounded space, flattening can be used in
the converse direction,
see 
Li--Shanmugalingam~\cite{LiShan} and
Durand-Cartagena--Li~\cite{DL1}, \cite{DL2}.
Thus we can flatten the sphere in Example~\ref{ex-Sn}, mapping
a point $c \notin \overline{\Om}$ to infinity,
e.g.\ $c=(0,\ldots,0,-1)$.
This would produce an example of a bounded domain in $\R^n$ 
in the spirit of the example by 
Llorente--Manfredi--Wu~\cite{LoMaWu},
where $\R^n$ necessarily needs to be equipped with some
weighted measure, if $p \ne n$.

In this particular
case, we can obtain the same result by 
using spherical inversion
as follows:
Let $X=\R^n$, equipped with the Lebesgue measure $dx$,
and let $Y$ be another copy of $\R^n$.
Let $\Phi: X\setm\{0\} \to Y\setm\{0\}$ be  given by 
\begin{equation}  \label{eq-inversion}
    \Phi(x)= \frac{x}{|x|^2},
    \quad \text{and thus} \quad
    \Phi^{-1}(y)= \frac{y}{|y|^2}.
\end{equation}
We extend $\Phi$ so that $\Phi(\infty)=0$ and $\Phi(0)=\infty$.
As we shall see, to preserve the \p-energy we need
to equip $Y$ with the measure $d\muh(y)=|y|^{2(p-n)}\, dy$.

Let $\Om \subset X$ be open and consider a function $u:\Om \to \eR$.
Let $\Omhat=\Phi(\Om)$ and $\uh=u \circ \Phi^{-1} : \Omhat \to \eR$.
As in Proposition~\ref{prop-Dp-equiv} we see that
\[
    \gh_{\uh}(y) = |x|^2 g_u(x) = \frac{g_u(x)}{|y|^2},
    \quad \text{where } y=\Phi(x),
\]
provided that $u \in \Dploc(\Om)$ and $\uh \in \Dploc(\Omhat)$.
It then follows
that
\[
         \int_{\Omhat} \gh_{\uh}^p \, d\muh(y) 
         =   \int_{\Om} (|x|^{2}g_u(x))^{p} |x|^{2(n-p)} |x|^{-2n}\, dx 
         =   \int_{\Om} g_u^p\, dx,
\] 
i.e.\ the \p-energy is indeed preserved
(and  $u \in \Dploc(\Om)$ if and only if $\uh \in \Dploc(\Omhat)$).

\begin{example}
To get a bounded example 
in the spirit of 
Llorente--Manfredi--Wu~\cite{LoMaWu},
we first shift their example and let
$\Om=\bigl\{(x_1,\ldots,x_n): x_n >\tfrac12\bigr\}$.
(Their sets $E_j$ and $A_j$
should also be shifted using the map 
$x \mapsto x+ \bigl(0,\ldots,0,\tfrac12\bigr)$.)
After performing the inversion~\eqref{eq-inversion},
we directly get a bounded
example, with $\Omhat$ being the unit ball in $\R^n$ 
centred at $(0,\ldots,0,1)$.
The sets $E_j$ and $A_j$ transfer as described above 
into $\widehat{E}_j=\Phi(E_j), \hat{A}_j=\Phi(A_j)\subset\bdry\Omhat$.

To fit within the standard framework as e.g.\ in
Heinonen--Kilpel\"ainen--Martio~\cite{HeKiMa},
we need $\muh$ to be a \p-admissible measure.
By \cite[p.\ 10]{HeKiMa}, 
this happens if and only if $2(p-n) > -n$, i.e.\
$p > n/2$, 
recovering a condition that
 we have encountered earlier.
(If $2(p-n) \le  -n$ then the weight is not locally 
integrable around the origin and thus cannot be \p-admissible.)

When $p=n>2$
the weight is $1$ and thus 
$(\R^n,d\muh)$ is the 
\emph{unweighted}~$\R^n$.
\end{example}

%%%%%%%%%%%%%%%%%%%%%%%%%%%%%%%%%%%%%%%


\begin{thebibliography}{99}

\bibitem{prime} \art{Adamowicz, T.,
        Bj\"orn, A., Bj\"orn, J. \AND Shan\-mu\-ga\-lin\-gam, N.}
        {Prime ends for domains in metric spaces}
        {Adv. Math.}{238}{2013}{459--505}


\bibitem{ABsuper} \art{Bj\"orn, A.}
        {Characterizations of \p-superharmonic
         functions on metric spaces}
        {Studia Math.} {169} {2005} {45--62}

\bibitem{ABkellogg} \art{Bj\"orn, A.}
         {A weak Kellogg property for quasiminimizers}
         {Comment. Math. Helv.} {81} {2006} {809--825}

\bibitem{ABclass} \art{\auth{Bj\"orn}{A}}
         {A regularity classification of boundary points
           for \p-harmonic functions and quasiminimizers}
        {J. Math. Anal. Appl.} {338} {2008} {39--47}

\bibitem{ABcluster} \art{\auth{Bj\"orn}{A}}
        {Cluster sets for Sobolev functions and quasiminimizers}
        {J. Anal. Math.} {112} {2010} {49--77}

\bibitem{ABgenreg} \arttoappear{\auth{Bj\"orn}{A}}
         {The Kellogg property and boundary regularity for 
         \p-harmonic functions with respect to the 
    Mazurkiewicz boundary and other compactifications}
  {Complex Var. Elliptic Equ.} \\ {\tt doi:10.1080/17476933.2017.1410799}.

\bibitem{BB} \art{\auth{Bj\"orn}{A} \AND \auth{Bj\"orn}{J}}
        {Boundary regularity for \p-harmonic functions and 
          solutions of the obstacle problem}
        {J. Math. Soc. Japan} {58} {2006} {1211--1232}

\bibitem{BB2} \art{Bj\"orn, A. \AND Bj\"orn, J.}
	{Approximations by regular sets and Wiener solutions in metric spaces}
	{Comment. Math. Univ. Carolin.} {48} {2007} {343--355}

\bibitem{BBbook} \book{\auth{Bj\"orn}{A} \AND \auth{Bj\"orn}{J}}
        {\it Nonlinear Potential Theory on Metric Spaces}
    {EMS Tracts in Mathematics {\bf 17},
        European Math. Soc., Z\"urich, 2011}

\bibitem{BBLeh1} \art{\auth{Bj\"orn}{A}, \auth{Bj\"orn}{J}
	\AND \auth{Lehrb\"ack}{J}}
        {Sharp capacity estimates for annuli in weighted $\R^n$ and metric
        spaces}{Math. Z.} {286} {2017} {1173--1215} 

\bibitem{BBS} \art{\auth{Bj\"orn}{A}, \auth{Bj\"orn}{J} 
    \AND \auth{Shanmugalingam}{N}}
        {The Dirichlet problem for \p-harmonic functions on metric spaces}
        {J. Reine Angew. Math.} {556} {2003} {173--203}

\bibitem{BBS2} \art{\auth{Bj\"orn}{A}, \auth{Bj\"orn}{J} 
    \AND \auth{Shanmugalingam}{N}}
        {The Perron method for \p-harmonic functions in metric spaces}
        {J. Differential Equations} {195} {2003} {398--429}

\bibitem{BBSdir} \art{\auth{Bj\"orn}{A}, \auth{Bj\"orn}{J} 
        \AND \auth{Shanmugalingam}{N}}
    {The Dirichlet problem for \p-harmonic functions with respect to
the Mazurkiewicz boundary, and new capacities}
    {J. Differential Equations} {259} {2015} {3078--3114}

\bibitem{BBSfincon} \art{\auth{Bj\"orn}{A}, \auth{Bj\"orn}{J} 
        \AND \auth{Shanmugalingam}{N}}
        {The Mazurkiewicz distance and sets that are finitely connected 
        at the boundary}
        {J. Geom. Anal.}{26}{2016}{873--897} 

\bibitem{BBSjodin} \art{\auth{Bj\"orn}{A}, \auth{Bj\"orn}{J} 
      \AND \auth{Sj\"odin}{T}}
  {The Dirichlet problem for  \p-harmonic functions with respect to arbitrary
  compactifications}
  {Rev. Mat. Iberoam.} {34} {2018} {1323--1360} 

\bibitem{BMartio} \art{ Bj\"orn, A. \AND Martio, O.}
	{Pasting lemmas and characterizations of boundary 
	regularity for quasiminimizers}
	{Results Math.} {55} {2009} {265--279}

\bibitem{JBpower} \art{\auth{Bj\"orn}{J}}
        {Poincar\'e inequalities for powers and products of admissible weights}
        {Ann. Acad. Sci. Fenn. Math.} {26}{2001}{175--188}

\bibitem{Bj02}  \art{\auth{Bj\"orn}{J}}
        {Boundary continuity for quasiminimizers on metric spaces}
        {Illinois J. Math.} {46} {2002} {383--403}

\bibitem{JBMatsue} \artin{\auth{Bj\"orn}{J}}
        {Wiener criterion for Cheeger \p-harmonic functions on metric spaces}
        {{\it Potential Theory in Matsue},
        Adv. Stud. Pure Math. {\bf 44}, pp. 103--115,
        Mathematical Society of Japan, Tokyo, 2006}

\bibitem{JB-pfine} \art{\auth{Bj\"orn}{J}}
        {Fine continuity on metric spaces}
        {Manuscripta Math.}{125}{2008}{369--381}

\bibitem{JBCalcVar} \art{\auth{Bj\"orn}{J}}
        {Necessity of a Wiener type condition for boundary regularity
          of quasiminimizers and nonlinear elliptic equations}
        {Calc. Var. Partial Differential Equations}
	{35} {2009} {481--496}

\bibitem{BMS} \art{\auth{Bj\"orn}{J}, \auth{MacManus}{P}
	\AND  \auth{Shanmugalingam}{N}}
        {Fat sets and pointwise boundary estimates for \p-harmonic functions
        in metric spaces}
        {J. Anal. Math.}{85}{2001}{339--369}

\bibitem{BonkKleiner02} \art{\auth{Bonk}{M} \AND \auth{Kleiner}{B}}
        {Rigidity for quasi-M\"obius group actions}
        {J. Differential Geom.} {61}{2002} {81--106} 

\bibitem{BuckleyHerronXie} \art{\auth{Buckley}{S. M}, \auth{Herron}{D}
    \AND \auth{Xie}{X}}
  {Metric space inversions, quasihyperbolic distance, and uniform spaces}
  {Indiana Univ. Math. J.} {57} {2008} {837--890}

\bibitem{DL1} \art{\auth{Durand-Cartagena}{E} \AND \auth{Li}{X}}
      {Preservation of \p-Poincar\'e inequality for large $p$ 
     under sphericalization and flattening}
      {Illinois J. Math.} {59} {2015} {1043--1069} 

\bibitem{DL2} \art{\auth{Durand-Cartagena}{E} \AND \auth{Li}{X}}
      {Preservation of bounded geometry under sphericalization and 
    flattening: quasiconvexity and $\infty$-Poincar\'e inequality}
  {Ann. Acad. Sci. Fenn. Math.} {42} {2017} {303--324}

\bibitem{Estep} \book{\auth{Estep}{D}}
  {Prime End Boundaries of Domains in Metric Spaces, and the Dirichlet Problem}
  {Ph.D. thesis, University of Cincinnati, Cincinnati, OH, 2015} 

\bibitem{Hansevi2} \art{\auth{Hansevi}{D}}
  {The Perron method for \p-harmonic functions in unbounded sets in 
   $\R^n$ and metric spaces}
  {Math. Z.} {288} {2018} {55--74}

\bibitem{HeKiMa} \book{Heinonen, J., Kilpel\"ainen, T.\ \AND Martio, O.}
        {Nonlinear Potential Theory of Degenerate Elliptic Equations}
        {2nd ed., Dover, Mineola, NY, 2006}

\bibitem{HeKo98} \art{Heinonen, J. \AND Koskela, P.}
	{Quasiconformal maps in metric spaces with controlled geometry}
	{Acta Math.} {181} {1998} {1--61}

\bibitem{HKSTbook} \book{\auth{Heinonen}{J}, \auth{Koskela}{P},
	\auth{Shanmugalingam}{N} \AND \auth{Tyson}{J. T}}
       {Sobolev Spaces on Metric Measure Spaces}
	{New Mathematical Monographs {\bf 27}, Cambridge Univ. Press,
        Cambridge, 2015}

\bibitem{Holo02} \art{\auth{Holopainen}{I}}
      {Asymptotic Dirichlet problem for the \p-Laplacian 
      on Cartan--Hadamard manifolds}
    {Proc. Amer. Math. Soc.} {130} {2002} {3393--3400}

\bibitem{HoloLV07} \art{\auth{Holopainen}{I}, \auth{Lang}{U} \AND
    \auth{V\"ah\"akangas}{A}}
    {Dirichlet problem at infinity on Gromov hyperbolic metric measure spaces}
    {Math. Ann.} {339} {2007} {101--134}

\bibitem{HoloV07} \art{\auth{Holopainen}{I} \AND
    \auth{V\"ah\"akangas}{A}}
  {Asymptotic Dirichlet problem on negatively curved spaces}
  {J. Anal.} {15} {2007} {63--110} 

\bibitem{karmazin2008} \book{\auth{Karmazin}{A. P}}
         {Quasiisometries, the Theory of Prime Ends and Metric Structures on Domains}
         {Surgut, 2008  (Russian)}

\bibitem{KeithZhong} \art{\auth{Keith}{S} \AND \auth{Zhong}{X}}
        {The Poincar\'e inequality is an open ended condition}
        {Ann. of Math.}{167}{2008}{575--599}

\bibitem{Kilp89} \art{\auth{Kilpel\"ainen}{T}} 
        {Potential theory for supersolutions of degenerate elliptic equations}
        {Indiana Univ. Math. J.} {38} {1989} {253--275}

\bibitem{KiMa02} \art{Kinnunen, J. \AND Martio, O.}
        {Nonlinear potential theory on metric spaces}
        {Illinois Math. J.} {46} {2002} {857--883}

\bibitem{KiSh01} \art{Kinnunen, J. \AND Shan\-mu\-ga\-lin\-gam, N.}
        {Regularity of quasi-minimizers on metric spaces}
        {Manuscripta Math.} {105} {2001} {401--423}

\bibitem{korte07} \art{\auth{Korte}{R}}
        {Geometric implications of the Poincar\'e inequality}
        {Results Math.}{50}{2007}{93--107}

\bibitem{KoMaSh} \art{\auth{Korte}{R}, \auth{Marola}{N} \AND 
        \auth{Shanmugalingam}{N}}
        {Quasiconformality, homeomorphisms between metric measure spaces 
        preserving quasiminimizers, and uniform density property}
        {Ark. Math.}{50}{2012}{111--134}

\bibitem{KoMc} \art{\auth{Koskela}{P} \AND \auth{MacManus}{P}}
        {Quasiconformal mappings and Sobolev spaces}
        {Studia Math.}{131}{1998}{1--17}

\bibitem{Laa} \artnopt{\auth{Laakso}{T}}
        {Ahlfors $Q$-regular spaces with arbitrary $Q>1$
         admitting weak Poincar\'e inequality}
        {Geom. Funct. Anal.}{10}{2000}{111--123};
        Erratum: \emph{Geom. Funct. Anal.} {\bf 12} (2002), 650.

\bibitem{LiShan} \art{\auth{Li}{X} \AND \auth{Shanmugalingam}{N}}
        {Preservation of bounded geometry under sphericalization and 
        flattening}
        {Indiana Univ. Math. J.}{64}{2015}{1303--1341}

\bibitem{LoMaWu} \art{\auth{Llorente}{J. G},
        \auth{Manfredi}{J. J} \AND \auth{Wu}{J.-M}}
	{\p-harmonic measure is not additive on null sets}
	{Ann. Sc. Norm. Super. Pisa Cl. Sci.} {4} {2005} {357--373}

\bibitem{nakki70} \art{\auth{N\"akki}{R}}
    {Boundary behavior of quasiconformal mappings in $n$-space}
    {Ann. Acad. Sci. Fenn. Ser. A I Math.} {484} {1970} {1--50}

\bibitem{Sh-rev} \art{Shanmugalingam, N.}
        {Newtonian spaces\textup{:} An extension of Sobolev spaces
        to metric measure spaces}
        {Rev. Mat. Iberoam.}{16}{2000}{243--279}

\bibitem{Sh-harm} \art{Shanmugalingam, N.}
        {Harmonic functions on metric spaces}
        {Illinois J. Math.}{45}{2001}{1021--1050}

\bibitem{vaisala} \book{\auth{V\"ais\"al\"a}{J}}
    {Lectures on $n$-dimensional Quasiconformal Mappings}
    {Lecture Notes in Math. {\bf 229},
    Springer, Berlin--Heidelberg, 1971}

\bibitem{wolff} \art{\auth{Wolff}{T. H}}
     {Gap series constructions for the \p-Laplacian}
     {J. Anal. Math.} {102} {2007} {371--394} 

\end{thebibliography}
\end{document}